\documentclass[11pt]{amsart}
\usepackage[english]{babel}
\usepackage[T1]{fontenc}
\usepackage[ansinew]{inputenc}
\usepackage{amsmath}
\usepackage{amsfonts}
\usepackage{ mathrsfs }
\usepackage{amssymb}
\usepackage{textcomp}
\usepackage{enumitem}
\usepackage[hidelinks]{hyperref}
\usepackage[arrow, matrix, curve]{xy}
\usepackage{comment}
\usepackage{color}
\usepackage{verbatim}
\usepackage{tikz}
\usetikzlibrary{decorations.pathreplacing}

\numberwithin{paragraph}{section}
\numberwithin{equation}{section}

\newtheorem{satz}{Theorem}[section]
\newtheorem{Thm}[satz]{Theorem}
\newtheorem{lem}[satz]{Lemma}

\newtheorem{prop}[satz]{Proposition}
\newtheorem{Prop}[satz]{Proposition}

\newtheorem{kor}[satz]{Corollary}
\theoremstyle{definition}
\newtheorem{defn}[satz]{Definition}
\newtheorem{bem}[satz]{Remark}
\newtheorem{theorem}{Theorem}[]
\newtheorem{thm}[theorem]{Theorem}
\newtheorem{koro}[theorem]{Corollary}

\newtheorem*{ack}{Acknowledgements}
\newtheorem*{term}{Terminology}

\newcommand{\N}{\mathbb{N}}
\newcommand{\Z}{\mathbb{Z}}

\newcommand{\Q}{\mathbb{Q}}
\newcommand{\R}{\mathbb{R}}

\newcommand{\G}{\mathbb{G}}

\newcommand{\vedge}{\land}

\newcommand{\Xan}{X^{\an}}

	\DeclareMathOperator{\an}{an}

	\DeclareMathOperator{\trop}{trop}
	\DeclareMathOperator{\relint}{relint}

	\DeclareMathOperator{\Spec}{Spec}

	\DeclareMathOperator{\Trop}{Trop}

	\DeclareMathOperator{\supp}{supp}

	\DeclareMathOperator{\id}{id}

	\DeclareMathOperator{\val}{val}

\DeclareMathOperator{\AS}{\mathcal{A}}

\DeclareMathOperator{\HS}{\mathcal{H}}

\DeclareMathOperator{\LS}{\mathcal{L}}

\DeclareMathOperator{\OS}{\mathcal{O}}

\DeclareMathOperator{\US}{\mathcal{U}}

\DeclareMathOperator{\XS}{\mathcal{X}}
\DeclareMathOperator{\YS}{\mathcal{Y}}

\DeclareMathOperator{\PB}{\mathbb{P}}

\def\quotient#1#2{\raise0.75ex\hbox{$\,#1$}\big/\lower0.75ex\hbox{$#2\,$}}

\title[Comparison of two notions of subharmonicity on non-archimedean curves]{Comparison of two notions of subharmonicity on non-archimedean curves}

\author[V.~Wanner]{Veronika Wanner}
\address{V. Wanner, Mathematik, Universit{\"a}t 
Regensburg, 93040 Regensburg, Germany}
\email{veronika.wanner@mathematik.uni-regensburg.de}

\thanks{The author was supported by the collaborative research center SFB 1085 'Higher Invariants' funded by the Deutsche Forschungsgemeinschaft.
 }

\begin{document}
\begin{abstract}
We show that a continuous function on the analytification of a smooth proper algebraic curve over a non-archimedean field is subharmonic in the sense of Thuillier if and only if it is psh, i.e.~subharmonic in the sense of Chambert-Loir and Ducros. 
This equivalence implies that the property psh for continuous functions is stable under pullback with respect to morphisms of curves.
Furthermore, we prove an analogue of the monotone regularization theorem on the analytification of $\mathbb{P}^{1}$ and Mumford curves using this equivalence.

\bigskip

\bigskip

\end{abstract}
\maketitle 
\tableofcontents

\section{Introduction}

Potential theory studies subharmonic functions and is a very old area of mathematics, which originates in the $18$th century.
This theory, and so in particular the theory of subharmonic functions, can be extended to non-archimedean analytic geometry. This is for example done by Baker and  Rumely in \cite{BR} for the non-archimedean projective line and by Thuillier in \cite{Th} for general curves. 
One is interested to develop this theory for non-archimedean spaces also in higher dimensions, i.e.~to develop a pluripotential theory analogue to the classical one.
 Ideas and concepts from pluripotential theory have been already introduced into the theory of non-archimedean analytic spaces by several authors as Zhang \cite{Zhang}, Boucksom, Favre, and Jonsson \cite{BFJ12}, Chambert-Loir and Ducros \cite{CLD}, and Gubler and K\"unnemann \cite{GK1, GK2}.
For example, Chambert-Loir and Ducros  defined in \cite{CLD} plurisubharmonic functions with the help of their real-valued differential forms and currents on Berkovich spaces.
Plurisubharmonicity is more precisely characterized by positivity of a special current corresponding to the given function.
 Their definition of plurisubharmonicity is analogous to the one in classical complex analysis.
As Thuillier's notion in the one dimensional case, their notion is locally analytic and works without any hypotheses on the characteristic.
Furthermore, they  introduced the Monge-Amp\`ere measure for plurisubharmonic functions that are locally approximable by smooth plurisubharmonic functions. 
This is a partial analogue of the complex Bedford-Taylor theory.
One would desire an analogue of this whole theory, and also a monotone regularization theorem in this setting would be worthwhile.
Moreover, we do not know if the notion of plurisubharmonicity is stable under pullback.

The subharmonic functions defined by Thuillier satisfy all required properties (cf. \cite[\S 3.2]{Th}).
In this present work,  we will show that for a continuous function on the analytification of a smooth proper algebraic curve his notion coincides with the one by Chambert-Loir and Ducros.

Let $K$ be an algebraically closed,  complete, non-archimedean, non-trivially valued field and $X$ be a smooth algebraic curve over $K$ with Berkovich analytification $\Xan$.
In Section 2, we summarize Thuillier's theory of subharmonic functions on open subsets of the analytification $\Xan$.
As in the complex potential theory, we use harmonic functions to define subharmonic functions. 
The definition of harmonic functions is related to skeletons of strictly affinoid domains. 
For every strictly affinoid domain $Y$ and every strictly semistable formal model $\YS$ of $Y$, which always exists by the semistable reduction theorem, one obtains a finite graph $S(\YS)$ in $Y$ with  a retraction map $\tau_{\YS}\colon Y\to S(\YS)$. 
Thuillier has introduced in \cite{Th} harmonic functions as continuous functions such that for every strictly affinoid domain $Y$ and every strictly semistable formal model $\YS$ of $Y$, we have  $f|_Y=F\circ \tau_{\YS}$ for a piecewise affine function $F$ on $S(\YS)$ such that the sum of outgoing slopes is zero everywhere in  $S(\YS)\backslash \partial Y$ for the Berkovich boundary $\partial Y$ of $Y$.
An upper semi-continuous function $f\colon W\to [-\infty,\infty)$ on an open subset $W$ of $\Xan$ is then called \emph{subharmonic} if and only if for every strictly affinoid domain $Y$ in $W$ and every harmonic function $h$ on $Y$ with $f\leq h$ on $\partial Y$, we have $f\leq h$ on $Y$.
In \cite{Th}, there is also a notion of smooth functions which we call \emph{lisse} (to distinct them from those defined by Chambert-Loir and Ducros). 
For every lisse function $f$, Thuillier defined a measure $dd^c f$ on $W$ that is positive if and only if $f$ is subharmonic. 
He showed that all the properties of subharmonic functions one would expect from the complex potential theory are satisfied. 
In particular, he verified the stability under pullback (see Proposition \ref{Prop Pullback}) and proved a version of the monotone regularization theorem with lisse subharmonic functions (see Proposition \ref{Prop Net}). 

In Section 3, we define plurisubharmonic functions in the sense of Chambert-Loir and Ducros on the analytification of an arbitrary algebraic variety $X$ of dimension $n$ over $K$. 
In the following, we just say psh.
To do this, we will first recall the sheaf of smooth differential forms $\AS_{X}^{p,q}$ on $\Xan$ with the differential operators $d'$ and $d''$. 
For an open subset $W$ of $\Xan$, we use the notation $\AS_{X,c}^{p,q}(W)$ for the sections of $\AS_{X}^{p,q}(W)$ with compact support in $W$. 
Every form $\omega\in \AS_{X,c}^{n,n}(W)$ defines a unique signed Radon measure with compact support on $W$, which we denote by $\mu_{\omega}$ (see Proposition \ref{Prop Radon}) and  write $\int_W f\wedge \omega:=\int_W f~d \mu_\omega$.
Hence, every continuous function $f\colon W\to \R$ on an open subset $W$ of $\Xan$ leads to a current on $\AS_{X,c}^{n-1,n-1}(W)$ given by
$$d'd''[f](\omega):=\int_Wf\wedge d'd''\omega.$$ 
In \cite{CLD}, a function $f$ is called \emph{psh} if and only if the induced current $d'd''[f]$ is positive.
Furthermore, we recall the definition of the Monge-Amp\`ere measure corresponding to a function that is locally approximable, i.e.~it is locally the difference of uniform limits of smooth psh functions. 
If $X$ is a curve and the Monge-Amp\`ere measure of a function $f$ is positive, then  $f$ is psh.
Note that if $X$ is proper and smooth, model functions on $\Xan$ are lisse, locally approximable and the  Monge-Amp\`ere measure $\text{MA}(f)$ coincides with Thuillier's measure $dd^cf$, which follows from results of Chambert-Loir--Ducros and Katz--Rabinoff -- Zureick-Brown.
A definition of model functions can be found in Definition \ref{Def model}.
These functions and these properties are key ingredients for the proofs in Section 4. 

In Section 4, we show the coincidence of the notions of subharmonic and psh for continuous functions on the analytification of algebraic curves.
We start by showing  in Proposition \ref{Thm lisse} that a lisse subharmonic function is psh.
Since the psh functions form a sheaf (see Proposition \ref{Prop psh sheaf}), it remains to show the assertion for an open neighborhood of every point. 
The complicated case is a point of type II.
 The strategy is to find an open neighborhood such that the lisse function $f$ can be written as an $\R$-linear combination of model functions. 
 One can deduce that the Monge-Amp\`ere measure of the restriction of $f$ coincides with the measure $dd^cf$ defined by Thuillier since it does for every model function, and the claim follows.

 Due to the analogue of Thuillier's monotone regularization theorem (see Proposition \ref{Prop Net}), we get the following theorem.
 \begin{thm}
 	Let $W$ be an open subset of $\Xan$ and $f\colon W\to [-\infty,\infty)$ be a subharmonic function, then $f$ is psh.
 \end{thm} 

We continue in Section 4 by proving the other implication for continuous functions, i.e.~the following theorem.
 \begin{thm}
 Let $W$ be an open subset of $\Xan$ and $f\colon W\to \R$ be a continuous function. If $f$ is psh, then $f$ is subharmonic.
 \end{thm}
 To obtain this result, we first show in Lemma \ref{Lemma Modelfct.} that for every non-subharmonic function $f$, there has to be a specific non-negative model function $g$ with compact support such that $\int f ~\text{MA}(g)=\int f~dd^c g<0$.
  The proof of the theorem is by contradiction, and we assume that such a function $g$ exists.
 We approximate $g$ by non-negative smooth functions $g_k$ with compact support such that  $\int f ~d\text{MA}(g)=\lim_{k\to\infty}\int f\wedge d'd''g_k$.
 Since $f$ is psh, the left hand side of the equation is non-negative, and so $f$ has to be subharmonic. 
The idea of the construction of these functions $g_k$ is based on \cite[Proposition 6.3.2]{CLD}  and we explain the modification in several steps.

In Section 5, we give some applications of this coincidence of the two notions of subharmonicity for continuous functions.
Since we know that subharmonicity in the sense of Thuillier is stable under pullback (see Proposition \ref{Prop Pullback}), we can deduce the following corollary from the theorems above.
  \begin{koro}
  Let $X,X'$ be smooth proper algebraic curves over $K$ and $\varphi\colon W'\to W$ be a morphism of $K$-analytic spaces for open subsets $W\subset \Xan$ and $W'\subset (X')^{\an}$. If a continuous function $f\colon W\to \R$ is psh on $W$, then  $\varphi ^*f$ is psh on $\varphi ^{-1}(W)$. 
  \end{koro}
 Furthermore, if $X$ satisfies some certain properties, e.g.~$X$ is the projective line or a Mumford curve, we obtain a  monotone regularization theorem in the setting of Chambert-Loir and Ducros  using the one in Thuillier's setting  (see Proposition \ref{Prop Net}).
  Note that we need all harmonic functions on $\Xan$  to be smooth in the sense of Chambert-Loir and Ducros to prove the statement.
  \begin{koro} Let $X$ be a smooth proper curve over $K$.
  	If the residue field $\widetilde{K}$ is algebraic over a finite field or $X$ is the projective line or a Mumford curve, then every continuous psh function $f\colon W\to \R$  on an open subset $W$ of $\Xan$ is locally psh-approximable.
  	More precisely, the sequences of smooth psh functions can be chosen monotonically decreasing.
  \end{koro}

\term{In this paper, let $K$ be an algebraically closed field endowed with a complete, 
non-archimedean, non-trivial absolute value $|~ |$.
We write $\val(K^\times):=-\log|K^\times|$ and $R:=\{x\in K\mid |x|\leq 1\}$ with maximal ideal $\mathfrak{m}\subset R$ and residue field $\widetilde{K}:=R/\mathfrak{m}$.
By suitably normalizing the absolute value $|~|$, we may assume that $\Z$, and so $\Q$, is contained in $\val(K^\times)$.
A variety over $K$ is an irreducible separated reduced scheme of finite type over $K$ and a curve is a $1$-dimensional variety over $K$.}

\ack{The author would like to thank Walter Gubler for very carefully reading drafts of this work and for the helpful discussions.
The author is also grateful to Philipp Jell for providing a simplification of the proof of  Proposition \ref{Prop lokal int} and further useful comments.  Finally, the author would like to thank the referee for their very
precise report and helpful suggestions.}

\section{Preliminaries on Thuillier's notion of subharmonic functions}
Let $X$ be a smooth algebraic curve over $K$ and $\Xan$ be its analytification.
In this section, we define subharmonic functions in the sense of Thuillier on  $\Xan$ and give some properties of these functions. 
Subharmonic functions are defined with the help of harmonic functions as in the complex case. 
We therefore recall the definition  of harmonic functions on strictly affinoid domains  from \cite[\S 2.3]{Th}.
For this definition, one needs the skeleton $S(\YS)$ corresponding to a strictly semistable formal model $\YS$ of a strictly affinoid domain $Y$, which has the structure of a finite graph. 

Furthermore, we study lisse (subharmonic) functions, which are crucial for the proofs in Section 4.
These functions are Thuillier's smooth functions, but note that they are totally different from those smooth functions introduced by Chambert-Loir and Ducros, which we define in Section 3. 
Hence, we use the term lisse instead of smooth here.

\begin{defn}
	
	A \emph{semistable $R$-curve} is an integral admissible formal $R$-curve $\YS$ whose special fiber has only ordinary double points as singularities.
	A semistable $R$-curve is called \emph{strictly semistable} if in addition the irreducible components of the special fiber are smooth.
	Let $Y$ be a \emph{strictly analytic domain} in $\Xan$, i.e.~a subset that has a locally finite covering by strictly affinoid domains, then we say that $\YS$ is a \emph{(strictly) semistable formal model} of $Y$  if $\YS$ is a (strictly) semistable proper formal $R$-curve with generic fiber $\YS_\eta$ isomorphic to $Y$.
	
\end{defn}
\begin{defn}
	Let $\YS$ be a strictly semistable proper formal $R$-curve, then there is a corresponding closed subset $S(\YS)$ of the generic fiber $\YS_\eta$ consisting of points of type II and III (see \cite[\S 1.4.4]{BerkovichSpectral} for the classification of points) that has the structure of a finite metric graph and has a canonical retraction map $\tau_{\YS}\colon \YS_\eta \to S(\YS)$ (cf. \cite[Th\'eor\`eme 2.2.10]{Th} or \cite[Chapter 3 \& 5]{BPR2}).
We call $S(\YS)$ the \emph{skeleton} of $\YS$ and the canonical metric is called the \emph{skeletal metric}.   
	Note that the irreducible components of the special fiber of $\YS$ correspond to vertices, which are all of type II, and
	the points in the intersection of two irreducible components correspond to the edges joining the two vertices.
	For every edge $e$ of $S(\YS)$, we have an isometry $\alpha_e\colon [a,b]\to e$ for a closed real interval $[a,b]$.
	A \emph{piecewise affine function} on $S(\YS)$ is a continuous function $F\colon S(\YS)\to \R$ such that  $F|_e\circ \alpha_e$ is piecewise affine for every edge $e$ of $S(\YS)$. 
We can define the \emph{outgoing slope} of a piecewise affine function $F$ on $S(\YS)$ at a point $x\in S(\YS)$ along a tangent direction $v_e$ at $x$ corresponding to an adjacent edge $e$ as $$d_{v_e}F(x):=\lim_{\varepsilon\to 0}(F|_e\circ \alpha_e)^\prime(\alpha_e^{-1}(x)+\varepsilon).$$
  
One obtains a finite measure  on $\Xan$ by  $$dd^c F:=\sum_{x\in S(\YS)}^{}(\sum_{v_e}d_{v_e} F(x))\delta_x,$$
where $e$ is running over all edges in $S(\YS)$ at $x$.
Note that $\sum_{v_e}d_{v_e} F(x)=0$ if $F$ is affine in $x$, and hence the sum is finite. 
Let $S_0$ be a finite subset of $S(\YS)$. We say that a  piecewise affine function $F$ on $S(\YS)$ is \emph{harmonic} on $S(\YS)\backslash S_0$, and write $F\in H(S(\YS), S_0)$, if  $dd^c F(x)=0$ for all $x\in S(\YS)\backslash S_0$.
\end{defn}
\begin{defn}\begin{enumerate}
		\item For a strictly semistable proper formal $R$-curve $\YS$, we define $$H(\YS):=\tau^*_{\YS}(H(S(\YS), \partial \YS_\eta)),$$
		where $\partial \YS_\eta$ is the Berkovich boundary (see \cite[\S 2.1.2]{Th}) of the generic fiber  $\YS_\eta$. 
\item Let $Y$ be a strictly affinoid domain in $\Xan$. Then there exists a strictly semistable $R$-curve $\YS$  and an isomorphism $\varphi\colon Y \to \YS_\eta$  (see \cite{BL2} and \cite[Th\'eor\`eme 2.3.8]{Th}), and we define the \emph{harmonic functions} 	on $Y$  as 
$$H(Y):=\varphi^*(H(\YS)).$$
Note that the definition is independent of $\varphi$ and $\YS$ \cite[Proposition 2.3.3]{Th}.\end{enumerate}
\end{defn}

\begin{defn}
	Let $W$ be an open subset of $\Xan$. Then $f\colon W \to [-\infty,\infty)$ is called \emph{subharmonic} if $f$ is upper semi-continuous,  $f\not\equiv-\infty$ on every connected component of $W$ and for every strictly $K$-affinoid domain $Y$ in $W$ and every harmonic function $h$ on $Y$,  we have 
	$$(f|_{\partial Y}\leq h|_{\partial Y})\Rightarrow (f|_Y\leq h).$$
\end{defn}

We recall some important properties of subharmonic functions from \cite[\S 3]{Th}.
\begin{Prop}
The subharmonic functions form a sheaf on $\Xan$.
\end{Prop}
\begin{proof}
See  \cite[Corollaire 3.1.13]{Th}.
\end{proof}

\begin{prop}\label{Prop Pullback}
	Let $X,X'$ be smooth algebraic curves, $\varphi\colon{(X')}^{\an}\to\Xan$ be a morphism of $K$-analytic spaces and $W$ be an open subset of $\Xan$. 
	If $f\colon W\to \R$ is a subharmonic function on $W$, then  $\varphi ^*f$ is a subharmonic function on $\varphi ^{-1}(W)$. 
\end{prop}
\begin{proof}
See  \cite[Proposition 3.1.14]{Th}.
\end{proof}
In the following, we define the class of lisse functions using the characterization in \cite[Proposition 3.2.4]{Th}. 
Note that we can replace compact strictly analytic domains by strictly affinoid domains looking at the proof of \cite[Proposition 3.2.4]{Th} and using \cite[Proposition 3.2.2]{Th}.

\begin{defn}\label{Def lisse} Let $W$ be an open subset of $\Xan$. A  function $f\colon W\to \R$ is called \emph{lisse} if
	for every strictly affinoid domain $Y\subset W$, there exists a strictly semistable formal $R$-curve $\YS$ such that $Y=\YS_\eta$ and $f|_Y=F\circ \tau_{\YS}$ for a piecewise affine function $F$ on $S(\YS)$.
\end{defn}
\begin{prop}\label{Prop Net} Let $W$ be an open subset of $\Xan$ and $f$ a subharmonic function on $W$. 
	For every relatively compact open subset  $W'$ of $W$,  
	there is a decreasing net $\langle f_j\rangle$ of lisse subharmonic functions converging pointwise to $f$ on $W'$.
\end{prop}
\begin{proof}
	See \cite[Th\'eor\`eme 3.4.2]{Th}.
\end{proof}
\begin{defn}\label{Bem dd^c} 
 Let $W$ be an open subset of $\Xan$ and  $f\colon W\to \R$ be a lisse function on $W$, then we denote by $dd^cf$ the unique real measure on $W$ that 
\begin{enumerate}
	\item has discrete support contained in the set of  type II and III points of $\Xan$, and
	\item coincides with $dd^c F$ on $Y$ whenever $f=F\circ \tau_{\YS}$ on a strictly affinoid domain $Y=\YS_\eta$ in $W$.
\end{enumerate}
This measure exists and is unique by \cite[Th\'eor\`eme 3.2.10]{Th}.
\end{defn}
With the help of this operator $dd^c$, one can define harmonic functions on open subsets in the Berkovich topology (see \cite[Corollaire 3.2.11]{Th}). 
\begin{defn}\label{Def harmonisch Berk}
A function $f\colon W\to \R$ on an open subset $W$ of $\Xan$ is called \emph{harmonic} if $f$ is lisse and $dd^cf=0$.
Harmonic functions  form a sheaf on $\Xan$, which we denote by $\HS_X$.
\end{defn}
This measure leads also to a further way to decide whether a lisse function is subharmonic or not.
\begin{prop} \label{Prop lisse subh}
A lisse function $f\colon W\to \R$ is subharmonic if and only if $dd^cf\geq 0$.
\end{prop}
\begin{proof}
	See \cite[Proposition 3.4.4]{Th}.
\end{proof}

\begin{prop}
	Let $Y$ be a connected strictly affinoid domain in $\Xan$ and $x\in Y\backslash\partial Y$ be a point of type II or III. 
	Then there exists a unique lisse function $g^Y_x$ on $\Xan$ such that \begin{enumerate}
		\item $g^Y_x$ is strictly positive on $Y\backslash\partial Y$ and equal to zero on $\partial Y \cup \Xan\backslash Y$;
		\item  $dd^c g^Y_x$ is supported on $\partial Y \cup \{x\}$ with  $dd^c g^Y_x=-\delta_x$ in a neighborhood of $x$;
		\item for every  harmonic function $h$ on $Y$ we have
		\begin{align}\label{Formel h(x)}
		h(x)=\int _{\partial Y} h ~dd^c g^Y_x.
		\end{align}
\end{enumerate}
\end{prop}
\begin{proof}
	See \cite[Proposition 3.3.7 \& Corollaire 3.3.9]{Th}.
\end{proof}
\begin{kor}\label{Kor Greenfct}
	Let $W$ be an open subset of $\Xan$. 
	A continuous function $f\colon W\to \R$ is subharmonic if and only if for every connected strictly affinoid domain $Y\subset W$ and every point $x\in Y\backslash \partial Y$ of type II or III, we have $$\int_W f~ dd^c g^Y_x\geq 0.$$
\end{kor}
\begin{proof}
	Consider a strictly affinoid domain $Y$ in $W$ and a harmonic function $f$ on $Y$ with $f\leq h$ on $\partial Y$. 
	As explained in \cite[Remarque 3.1.10]{Th}, we may assume $f=h$ on $\partial Y$, and without loss of generality $Y$ is connected.
	For every $x\in Y\backslash\partial Y$ of type II or III,   Equation $(\ref{Formel h(x)})$ implies that
	\begin{align*}
	\int_W f~ dd^c g^Y_x& = \int_{\partial Y} f~ dd^c g^Y_x - f(x) = \int_{\partial Y} h~ dd^c g^Y_x - f(x)=h(x)-f(x).
	\end{align*}
	Hence, we have $\int_W f~ dd^c g^Y_x\geq 0$ for every type II or III point  in $Y$ if and only if $h(x)\geq f(x)$ for every type II or III point  in $Y$. 
	Since $f$ and $h$ are continuous and the type II and III points are dense, the last is equivalent to $h-f \geq 0$ on $Y$.
\end{proof}
\section{Preliminaries on subharmonic functions in the sense of Chambert-Loir and Ducros}
In this section, let $X$ be an $n$-dimensional algebraic variety  over $K$.
 We define psh functions on open subsets of $\Xan$ in the sense of Chambert-Loir and Ducros \cite{CLD}.
For the definition of psh functions, we need the sheaf of smooth real-valued differential forms on Berkovich analytic spaces introduced by them. 
We follow \cite{Gubler} for the definition of these forms $\mathcal{A}_X^{p,q}$ on the analytification $\Xan$ of an algebraic variety $X$.
One can also define currents which are linear functionals on $\mathcal{A}_{X,c}^{p,q}(W)$ for an open subset $W$ of $\Xan$ satisfying a continuity property, 
where $\mathcal{A}_{X,c}^{p,q}(W)$ denotes  the sections of $\mathcal{A}_X^{p,q}(W)$ with compact support.
We refer to \cite[\S 4.2]{CLD} or \cite[\S 6]{Gubler} for the precise definition of currents.
Subsequently,  we will give a definition of the Monge-Amp\`ere measure corresponding to locally approximable functions, and give some properties of it. These properties are needed in Section 4 for the proofs of our main theorems.

We start with recalling $(p,q)$-superforms on open subsets of  $\R^r$,  which were originally defined by Lagerberg in \cite{Lagerberg}. 
This theory of superforms leads to superforms on polyhedral complexes developed in \cite{CLD}.
With the help of the theorem of Bieri-Groves one can define real-valued differential forms on $\Xan$.

\begin{defn} \label{defnd}
\begin{enumerate} [itemindent =*, leftmargin=0mm]
\item
For an open subset $U \subset \R^r$ denote by $\AS^{p}(U)$ the space of smooth real differential forms of degree $p$. 
The space of \emph{superforms of bidegree} $(p,q)$ on $U$ is defined as
\begin{align*}
\AS^{p,q}(U)
:= \AS^p(U)\otimes_{C^\infty(U)}\AS^q(U)= \AS^p(U)\otimes_\R \Lambda^q {\R^r}^*= \Lambda^p {\R^r}^* \otimes_\R \AS^q(U).
\end{align*}
\item
There are  \emph{differential operators} 
\begin{align*}d' \colon \AS^{p,q} (U) =\AS^p(U)\otimes_\R \Lambda^q {\R^r}^*\rightarrow \AS^{p+1}(U)\otimes_\R \Lambda^q {\R^r}^*= \AS^{p+1,q}(U)\\
d'' \colon \AS^{p,q} (U) = \Lambda^p {\R^r}^* \otimes_\R \AS^q(U) \rightarrow \Lambda^p {\R^r}^* \otimes_\R \AS^{q+1}(U)= \AS^{p,q+1}(U)
\end{align*}
 that are given by $ D \otimes (-1)^q \id$ and $(-1)^p \id \otimes D$, where $D$ is the usual exterior derivative.

\item 
There is a wedge product  
\begin{align*}
\wedge: \AS^{p,q}(U) \times \AS^{p',q'}(U) &\rightarrow \AS^{p+p',q+q'}(U), \\
(\alpha \otimes \psi, \beta \otimes \nu) &\mapsto (-1)^{p' q} \alpha \vedge \beta \otimes \psi \wedge \nu ,
\end{align*}
that is, up to sign, induced by the usual wedge product.

\item 
We have the following canonical involution $J\colon \AS^{p,q}(U)\to \AS^{q,p}(U)$ given by
\begin{align*}
\alpha =\sum\limits_{I,J}^{}\alpha_{IJ}d'x_I\wedge d''x_J\mapsto (-1)^{pq} \sum\limits_{I,J}^{}\alpha_{IJ}d'x_J\wedge d''x_I,
\end{align*}
where $d'x_I\wedge d''x_J:=(dx_{i_1}\otimes\ldots\otimes dx_{i_p})\otimes_{\R}(dx_{j_1}\otimes\ldots\otimes dx_{j_q})$ for $I=\{i_1,\ldots,i_p\}$ and $J=\{j_1,\ldots,j_q\}$.

\end{enumerate}
For all $p,q$ the functor $U\mapsto \AS^{p,q}(U)$ defines a sheaf on $\R^r$ and we have $\AS^{p,q}=0$ if $\max(p,q)>r$.
\end{defn}

\begin{defn}
\begin{enumerate}
\item A \textit{polyhedron} in $\mathbb{R}^r$ is the intersection of finitely many half-spaces $H_i:=\{w\in \mathbb{R}^r|\langle u_i,w\rangle\leq c_i\}$ with $u_i\in {\mathbb{R}^r}^*$ and $c_i\in \R$. 
\item A \textit{polyhedral complex} $\mathscr{C}$ in $\mathbb{R}^r$ is a finite set of polyhedra in $\mathbb{R}^r$ satisfying the following two properties:
\begin{enumerate}
\item If $\tau$ is a closed face of a polyhedra $\sigma\in \mathscr{C}$, then $\tau\in \mathscr{C}.$
\item  If $\sigma,\tau\in \mathscr{C}$, then $\sigma\cap \tau$ is a closed face of both.
\end{enumerate}
\end{enumerate}\end{defn}

\begin{defn}
Let $\mathscr{C}$ be a polyhedral complex in $\mathbb{R}^r$.
\begin{enumerate}
\item  We say that $\mathscr{C}$  is of \textit{dimension} $d$ if the maximal dimension of its polyhedra is $d$.
 A polyhedral complex $\mathscr{C}$  is called \textit{pure dimensional of dimension} $d$ if every maximal polyhedron in $\mathscr{C}$ has dimension $d$.
\item The \textit{support} $\vert \mathscr{C}\vert $ of $\mathscr{C}$ is the union of all polyhedra in $\mathscr{C}$.
\item For $\sigma\in \mathscr{C}$, we denote by $\relint(\sigma)$ the relative interior of $\sigma$, by $\mathbb{A}_\sigma$  the affine space that is spanned by $\sigma$ and by $\mathbb{L}_\sigma$ the corresponding linear subspace of $\mathbb{R}^r$.
\end{enumerate}
\end{defn}

\begin{defn}
Let  $\mathscr{C}$ be a polyhedral complex and $\Omega$ be an open subset of $\vert \mathscr{C}\vert$. Then a \textit{superform} $\alpha\in \AS^{p,q}(\Omega)$ \textit{of bidegree} $(p,q)$ on $\Omega$ is given by a superform $\alpha'\in \AS^{p,q}(V)$ where $V$ is an open subset of $\mathbb{R}^r$ with $V\cap \vert \mathscr{C}\vert =\Omega$. 
Two forms  $\alpha'\in \AS^{p,q}(V)$ and  $\alpha''\in \AS^{p,q}(W)$ with  $V\cap \vert \mathscr{C}\vert =W\cap \vert \mathscr{C}\vert=\Omega$ define the same form in $ \AS^{p,q}(\Omega)$ if we have for each $\sigma\in \mathscr{C}$
$$\langle\alpha'(x);v_1,\ldots,v_p,w_1,\ldots,w_q\rangle=\langle\alpha''(x);v_1,\ldots,v_p,w_1,\ldots,w_q\rangle$$  for all $x\in \sigma\cap \Omega$, $v_1,\ldots,v_p,w_1,\ldots,w_q\in \mathbb{L}_\sigma$. 
If this is true, we say that the restrictions $\alpha'|_{\sigma}$ and $\alpha''|_{\sigma}$ agree. If $\alpha\in \AS^{p,q}(\Omega)$ is given by $\alpha'\in \AS^{p,q}(V)$, we write $$\alpha'|_{\Omega}=\alpha.$$
\end{defn}
\begin{bem} Let  $F\colon\mathbb{R}^{r'}\to \mathbb{R}^r$ be an affine map.
	 If $\mathscr{C}'$ is a polyhedral complex of $\mathbb{R}^{r'}$ and $\mathscr{C}$ a polyhedral complex of $\mathbb{R}^{r}$ with $F(\vert \mathscr{C}'\vert)\subset \vert \mathscr{C}\vert$, then the pullback $F^*\colon \AS^{p,q}(\vert \mathscr{C}\vert)\to \AS^{p,q}(\vert\mathscr{C'}\vert)$ is well-defined  
and compatible with the differential operators $d'$ and $d''$. 
In particular, we have the operators 
 $d'$ and $d''$ on $\AS^{p,q}(\vert \mathscr{C}\vert)$ given by the restriction of the corresponding operators on $\AS^{p,q}(\mathbb{R}^r)$.
\end{bem}

\begin{defn} We set $\G^r_m:=\Spec K[T_1^{\pm 1},\ldots,T_r^{\pm 1}]$.
	 Recall that its Berkovich analytification $\G^{r,\an}_m$ is the set of all multiplicative seminorms on $K[T_1^{\pm 1},\ldots,T_r^{\pm 1}]$ extending the given absolute value $|~|$ on $K$ endowed with the coarsest topology such that $\G^{r,\an}_m\to \R;~\rho\mapsto \rho(f)$ is continuous for every $f\in K[T_1^{\pm 1},\ldots,T_r^{\pm 1}]$. 
	 We define 
	$$\trop\colon \G^{r,\an}_m\to \R ^r; ~\rho\mapsto (-\log(\rho(T_1)),\ldots,-\log(\rho(T_r))).$$
	Note that $\trop$ is a proper map in the sense of topological spaces.
	
	Let $U$ be an open affine subset of $X$ and let $\varphi\colon U\to \G^r_m$ be a closed embedding, then we set $\varphi_{\trop}:=\trop\circ \varphi ^{\an}$ and $\Trop_{\varphi}(U):=\varphi_{\trop}(U^{\an})$. 	 
\end{defn}	
\begin{Thm}[Bieri-Groves]\label{BG}
For every open affine subset $U$ of $X$ and every closed embedding $\varphi\colon U\to \G^r_m$, the set $\Trop_{\varphi}(U)$ is the support of an $\R$-affine polyhedral complex of pure dimension $n=\dim(X)$. 
\end{Thm}
\begin{proof}
	See \cite[Theorem A]{BG}.
\end{proof}
\begin{defn}
Let $U, U'$ be open affine subsets of $X$ and  $\varphi\colon U\to \G^r_m$ and $\varphi'\colon U'\to \G^{r'}_m$  be closed embeddings.
We say that $\varphi'$ \emph{refines} $\varphi$ if $U'\subset U$ and there is an affine homomorphism (i.e.~group homomorphism composed with a multiplicative translation) $\psi\colon \G^{r'}_m\to \G^{r}_m$ of multiplicative tori such that $\varphi=\psi\circ \varphi'$. 
This homomorphism induces an integral affine map $\Trop(\psi)\colon \R^{r'}\to \R^{r}$ such that $\varphi_{\trop} =\Trop(\psi)\circ \varphi'_{\trop}$.

We call an open affine subset $U$ of $X$ \emph{very affine} if $\OS(U)$ is generated as a $K$-algebra by $\OS(U)^\times$.
If $U$ is very affine, then there exists a canonical (up to multiplicative translation) closed embedding $\varphi_U\colon U\to \G^r_m$ which refines all other closed embeddings $\varphi\colon U\to \G^{r'}_m$ (see \cite[4.12]{Gubler}).
For the canonical embedding $\varphi_U$, we use the notations $\trop_U:=(\varphi_U)_{\trop}$ and $\Trop(U):=\trop_U(U^{\an})$.
\end{defn}	 
	\begin{defn}
	Let $W$ be an open subset of $\Xan$. 
	A \emph{tropical chart} $(V,\varphi_U)$ of $W$ consists of the canonical closed embedding $\varphi_U\colon U\to \G^r_m$ of a very affine open subset $U$ of $X$ and an open subset $V$ of $W$ that is of the form $V=\trop_U^{-1}(\Omega)$ for an open subset $\Omega$ of $\Trop(U)$. 
	 
	We say that  $(V',\varphi_{U'})$ is a \emph{tropical subchart} of $(V,\varphi_U)$ if $V'\subset V$ and $U'\subset U$. 
	In this situation, $\varphi_{U'}$ refines 	$\varphi_{U}$.
	\end{defn}
	
\begin{bem}
Theorem \ref{BG} allows us to consider a superform $\alpha\in \AS^{p,q}_{\Trop(U)}(\mathrm{trop}_U(V))$ for a tropical chart $(V,\varphi_U)$ of $X^{\mathrm{an}}$.
 Let  $(V',\varphi_{U'})$  be another tropical chart of $X^{\mathrm{an}}$, then $(V\cap V' ,\varphi_{U\cap U'})$ is a tropical subchart of both by \cite[Proposition 4.16]{Gubler} with $\varphi_{U\cap U'}=\varphi_{U}\times\varphi_{U'}$.
  We get a canonical homomorphism $\psi_{U,U\cap U'}\colon \mathbb{G}_m^{r+r'}\to \mathbb{G}_m^r $ of the underlying tori with $$\varphi_U=\psi_{U,U\cap U'}\circ \varphi_{U\cap U'}$$ on $U\cap U'$ and an associated affine map $\mathrm{Trop}(\psi_{U,U\cap U'})\colon \mathbb{R}^{r+r'}\to \mathbb{R}^r$ such that $$\mathrm{trop}_{U}=\mathrm{Trop}(\psi_{U,U\cap U'})\circ \mathrm{trop}_{U\cap U'}$$ and the tropical variety $\mathrm{Trop}(U\cap U')$ is mapped onto  $\mathrm{Trop}(U)$ (see \cite[5.1]{Gubler}). 
  We define the \emph{restriction} of $\alpha$ to $\mathrm{trop}_{U\cap U'}(V\cap V')$ as 
$$\mathrm{Trop}(\psi_{U,U\cap U'})^*\alpha\in  \AS^{p,q}_{\Trop(U\cap U')}(\mathrm{trop}_{U\cap U'}(V\cap V'))$$ and write $\alpha|_{V\cap V'}$.
\end{bem}
\begin{defn}\label{D forms} Let $X$ be an algebraic variety over $K$ and $W$ be an open subset of $\Xan$. 
An element of $\AS_X^{p,q}(W)$ is given by a family $(V_i,\varphi_{U_i},\alpha_i)_{i\in I}$ such that 
\begin{enumerate}
\item for all $i\in I$ the pair $(V_i,\varphi_{U_i})$ is a tropical chart and $W=\bigcup_{i\in I}V_i$;
\item for all $i\in I$ we have $\alpha_i\in \AS^{p,q}_{\Trop(U_i)}(\mathrm{trop}_{U_i}(V_i))$;
\item for all $i,j\in I$ the restrictions $\alpha_i|_{V_i\cap V_j}= \alpha_j|_{V_i\cap V_j}$ agree;
\end{enumerate}

If $\alpha'$ is another differential form of bidegree $(p,q)$ on $W$ given by a triple $(V'_j,\varphi_{U'_j},\alpha'_j)_{j\in J}$, then we consider $\alpha$ and $\alpha'$ as the same differential form if and only if $$\alpha_i|_{V_i\cap V'_j}=\alpha'_j|_{V_i\cap V'_j}$$ for every $i\in I$ and $j\in J$.

Then  $W\mapsto\AS_X^{p,q}(W)$ defines a sheaf on $\Xan$, which we denote by  $\AS_X^{p,q}$ and we write $\AS_{X,c}^{p,q}(W)$ for the sections with compact support in $W$.
If the space of definition is clear, we often just use the notations $\AS^{p,q}$ and $\AS_c^{p,q}$. 
By Theorem \ref{BG}, we have $\AS_X^{p,q}=0$ if $\max(p,q)>n=\dim(X)$.

The differentials $d'$ and $d''$ and the wedge product carry over. Moreover, for every open subset $W$ of $\Xan$ there is a non-trivial integration map $\int\colon \AS^{n,n}_{c}(W)\to \R$ which is compatible with pullback.

Note that smooth differential forms of bidegree $(0,0)$ are well-defined continuous functions.
\end{defn}	

\begin{defn}\label{Def glatt}
	Let $W$ be an open subset of $\Xan$. 
	A function $f\colon W\to \R$ is called \emph{smooth} if $f\in\mathcal{A}^{0,0}(W)$. 
	We use the notations $C^\infty(W):=\mathcal{A}^{0,0}(W)$ and $C_c^\infty(W):=\mathcal{A}_c^{0,0}(W)$.
\end{defn}
Next to $(0,0)$-forms, we also work a lot with $(n,n)$-forms.
It is important to know the following fact about these forms. 
\begin{prop}\label{Prop Radon}
	Let $W$ be an open subset of $\Xan$ and $\alpha\in \mathcal{A}^{n,n}(W)$.
	 Then there is a unique signed Radon measure $\mu_\alpha$ on $W$ such that $\int_W f~d\mu_\alpha=\int_W f\alpha$ for every $f\in C^\infty_c(W)$. 
	 If $\alpha$ has compact support on $W$, so has $\mu_\alpha$ and $|\mu_\alpha|(W)<\infty$.
\end{prop}
\begin{proof}
	See \cite[Proposition 6.8]{Gubler}.
\end{proof}
\begin{defn} Let $U$ be an open subset of $\R^r$.
	 A superform $\alpha\in\AS^{p,p}(U)$ is called \emph{strongly positive} if there exist finitely many superforms  $\alpha_{j,s}$ of type $(0,1)$ and non-negative smooth functions $f_s$ on $U$ such that 
$$\alpha=\sum_{s}f_s\alpha_{1,s}\wedge J(\alpha_{1,s})\wedge\ldots\wedge \alpha_{p,s}\wedge J(\alpha_{p,s}).$$

Let $\mathscr{C}$ be a polyhedral complex and $\Omega$ an open subset of $|\mathscr{C}|$. 
A superform $\alpha\in \AS^{p,p}(\Omega)$ is called \emph{strongly positive} if there is a polyhedral decomposition of $\mathscr{C}$ such that the restriction of $\alpha$ to  $\text{relint}(\sigma)\cap \Omega$  is strongly positive for every polyhedron $\sigma\in\mathscr{C}$. 

For an open subset $W$ of $\Xan$, a smooth form $\omega\in \AS^{p,p}(W)$ is called $\emph{strongly positive}$ if for every point $x$ in $W$ there is a tropical chart $(V,\varphi_U)$ with $x\in V$ such that $\omega= \alpha\circ \trop_U$ on $V$ for a strongly positive form $\alpha\in \AS_{\Trop(U)}^{p,p}(\trop_U(V))$.

Note that for forms of type $(0,0)$, $(1,1)$, $(n-1,n-1)$ and $(n,n)$ the notion of strongly positivity defined here coincides with the other positivity notions from \cite[\S 5.1]{CLD}. 
Thus, we just say that a smooth form $\omega$  in $\AS^{p,p}(W)$ is \emph{positive} if it is  of one of these types and it is strongly positive. 

A smooth function $f$ is  (strongly) positive as a form if and only if $f\geq 0$. \end{defn}

\begin{defn} Let $W$ be an open subset of $\Xan$. We call a function $f\colon W \to[-\infty,\infty]$ \emph{locally integrable} if $f$ is integrable with respect to every measure $ \mu_\alpha$ associated to  a form $\alpha\in \mathcal{A}^{n,n}_c(W)$. We write $\int_W f\wedge \alpha:=\int_W f~ d\mu_\alpha.$

 Then for every locally integrable (e.g. continuous) function $f\colon W \to[-\infty,\infty]$ one can define a current in the sense of \cite[\S 4.2]{CLD} by 
	\begin{align*}
	d'd''[f]\colon \AS^{n-1,n-1}_c(W)\to \mathbb{R},~\alpha\mapsto \int_{W} f\wedge d'd'' \alpha.
	\end{align*}
\end{defn}

\begin{defn}\label{Def psh}
Let $W$ be an open subset of $\Xan$.
	A locally integrable function $f\colon W \to[-\infty,\infty)$ is called \emph{psh} if  $f$ is upper semi-continuous and $d'd''[f]$ is a positive current, i.e.~$d'd''[f](\omega)\geq 0$ for all positive forms $\omega\in \AS_c^{n-1,n-1}(W)$.
\end{defn}
Note that we do not require that a psh function has to be continuous, contrary to \cite[D\'efinition 5.5.1]{CLD}.
\begin{prop}\label{Prop partition} Let $W$ be a paracompact open subset of $\Xan$ and $(V_i)_{i\in I}$ be an open covering of $W$. 
	Then there are smooth non-negative functions $(\eta_{j})_{j\in J}$ with compact support on $W$ such that 
	\begin{enumerate}
		\item the family $(\supp(\eta_j))_{j\in J}$ is locally finite on $W$;
		\item we have $\sum\nolimits _{j\in J}\eta_j\equiv 1$ on $W$;
		\item for every $j\in J$, there is a $i(j)\in I$ such that $\supp(\eta_j)\subset V_{i(j)}$.
		\end{enumerate}
	We call $(\eta_{j})_{j\in J}$ a	\emph{partition of unity} subordinated to the open covering $(V_i)_{i\in I}$.
\end{prop}	
\begin{proof} See \cite[Proposition 5.10]{Gubler}.
\end{proof}
\begin{bem}\label{Bem 1}
	Note that every open subset $W$ of $\Xan$ is paracompact if $X$ is a curve \cite[Theorem 4.2.1 \& 4.3.2]{BerkovichSpectral}.
\end{bem}

\begin{prop}\label{Prop psh sheaf}
The psh functions  form a sheaf on $\Xan$.
\end{prop}
\begin{proof}

Let  $W$ be an open subset of $\Xan$ and $(W_i)_{i\in I}$ be an open covering of $W$. Consider a function $f\colon W\to [-\infty,\infty)$. 
If $f$ is psh on $W$, the restrictions $f|_{W_i}$ are clearly psh for every $i\in I$.

Assume that $f|_{W_i}$ is psh on $W_i$ for every $i\in I$ and show that $f$ is then psh on $W$.
Consider $\omega\in \AS^{n,n}_c(W)$ and let $V$ be a paracompact open neighborhood of $\supp(\omega)$ in $W$, which we can find by \cite[2.1.5 \& Lemme 2.1.6]{CLD}. 
Then $\omega\in \AS^{n,n}_c(V)$, the family $(V_i:=W_i\cap V)_{i\in I}$ defines an open covering of $V$ and the restrictions $f|_{V_i}$ are psh for every $i\in I$.  
Let $(\eta_j)_{j\in J}$ be a partition of unity subordinated to the covering $(V_i)_{i\in I}$ (see Proposition \ref{Prop partition}).
Then for every $j\in J$, we have $\eta_j\omega\in \AS_c^{n,n}(V_{i(j)})$.
Since $f|_{V_{i(j)}}$ is psh, the integral  $\int _{V_{i(j)}} f \wedge \eta_j\omega$ has to be finite for every $j\in J$.
Furthermore, we can write $\omega=\sum_{j\in J}^{}\eta_j\omega$ on $V$, where  this sum has to be finite as $\omega$ has compact support on $V$. 
Hence, 
\begin{align*}
	\int_W f \wedge \omega = 	\int_V f \wedge \omega = \sum_{j\in J} \int _{V_{i(j)}} f \wedge \eta_j\omega
\end{align*}
has to be finite as well, i.e.~$f$ is locally integrable. 

Now, consider a positive form $\omega\in\AS^{n-1,n-1}_c(W)$. 
Again, we work over the paracompact open neighborhood $V$ of $\supp(\omega)$ in $W$ and consider $\omega$ as a form in $\AS^{n-1,n-1}_c(V)$.
As above, we have a partition of unity $(\eta_j)_{j\in J}$ subordinated to the covering $(V_i)_{i\in I}$, and $\omega=\sum_{j\in J}\eta_j\omega$, where the sum is finite since $\omega$ has compact support.
Then $\eta_j \omega\in \AS_c^{n-1,n-1}(V_{i(j)})$ and $\eta_j\omega$ is a  positive form  for every $j\in J$.
Thus, 
\begin{align*}
d'd''[f|_{V}](\omega)= d'd''[f](\sum\nolimits_j \eta_j\omega)= \sum\nolimits_j d'd''[f](\eta_j \omega)\geq 0
\end{align*}
by linearity and the fact that $f|_{V_{i(j)}}$ is psh for every $j\in J$.
Since $d'd''[f](\omega)=d'd''[f|_{V}](\omega)$, the function $f$ is psh on $W$.
\end{proof}

Next, we translate a very useful characterization of smooth psh functions from \cite[Lemme 5.5.3]{CLD} to our setting.
\begin{prop}\label{Prop smooth psh}
Let $W$ be an open subset of $\Xan$.	
A smooth function $f\colon W\to \R$ is psh if and only if for every $x\in W$ there is a tropical chart $(V,\varphi_U\colon U\to \G^r_m)$ of $W$ with $x\in V$ such that $f=\psi\circ \trop_U$ on $V$ for a smooth function $\psi\colon \R^r\to \R$ whose restriction $\psi|_\sigma$ to every polyhedron $\sigma$ in $\R^r$ with $\sigma\subset \trop_U(V)$ is convex.
\end{prop}
\begin{proof} First, assume that $f$ is psh on $W$, i.e.~$d'd''[f]$ defines a positive current on $W$.
		Since $f$ is smooth, we can find for every $x\in W$  a tropical chart $(V,\varphi_U\colon U\to \G^r_m)$ in $W$ with $x\in V$ such that $f=\psi \circ\trop_U$ on $V=\trop_U^{-1}(\Omega)$ for a smooth function $\psi\colon \R ^r\to \R$ and an open subset $\Omega$ of $\Trop(U)$.
     	We choose a compact neighborhood $B$ of $\trop_U(x)$ in $\Omega$. 
     	Then the preimage $Y:=\trop_U^{-1}(B)$ under the proper map $\trop_U$ is a compact analytic domain. 
     	The restriction of the current $d'd''[f]$ to this compact analytic domain $Y$ is still positive. 
     	Applying \cite[Lemme 5.5.3]{CLD} to $f=\psi\circ \trop_U\colon Y\to \R$, we know that for every polyhedron $\Delta$ in $\R^r$ with $\Delta\subset B$  the restriction $\psi|_\Delta$ is convex.
     	Now, let $\Omega'$ be an open neighborhood of $\trop_U(x)$ in $B$ and consider the open neighborhood  $V':=\trop_U^{-1}(\Omega')$ of $x$ in $\trop_U^{-1}(B)$. 
     	Then the pair $(V',\varphi_U)$ is a tropical chart in $W$ that contains $x$ and $f|_{V'}= \psi\circ \trop_U$, where $\psi$ is smooth.
     	Consider a polyhedron $\sigma$ in $\R^r$ with $\sigma\subset \Omega'=\trop_U(V')$. 
     	Then $\sigma\subset B$, and so  $\psi|_{\sigma}$ is convex.
     
Next, we assume that there is for every $x\in W$ a tropical chart $(V,\varphi_U)$ of $W$ with $x\in V$ such that $f=\psi\circ \trop_U$ on $V$ for a function $\psi$ satisfying i) and ii). 
Our goal is to show that $f$ is psh in a neighborhood of $x$. 
As above, we choose a compact neighborhood $B$ of $\trop_U(x)$ in $\trop_U(V)$ and set $Y:=\trop_U^{-1}(B)$. 
Since every polyhedron $\Delta$ in $\R^r$ with $\Delta\subset B$ is contained in $\trop_U(V)$, the restriction $\psi|_{\Delta}$ is convex by ii).
 Again \cite[Lemme 5.5.3]{CLD} tells us that $f$ is psh on the compact analytic domain $Y$. Choosing $(V',\varphi_U)$ as above, we obtain a tropical chart of $W$ with $x\in V'$. 
Due to $V'\subset Y$, the function $f$ is psh on the open neighborhood $V'$ of $x$. 
Psh functions form a sheaf by Proposition \ref{Prop psh sheaf}, and so the claim follows.
\end{proof}
\begin{bem}\label{Pullback Bem}
Let  $X$ and $X'$ be algebraic varieties over $K$,
$\varphi\colon W'\to W$ a morphism of analytic spaces for open subsets $W\subset \Xan$ and $W'\subset (X')^{\an}$, and $f\colon W\to [-\infty,\infty)$ a psh function on $W$. 
If $f$ is smooth, it follows directly from the definition of smooth functions by Chambert-Loir and Ducros \cite[3.1.3]{CLD} and \cite[Proposition 7.2]{Gubler}  that $\varphi ^*f$ is smooth on $\varphi^{-1}(W)$.
Furthermore, \cite[Proposition 7.2]{Gubler} and Proposition \ref{Prop smooth psh} imply that $\varphi ^*f$ is also psh on $\varphi^{-1}(W)$. 
If $f$ is not smooth, it is not clear whether $\varphi ^*f$ is psh or not.
At the end of this paper, we will answer this question positively for a continuous psh function $f$ on an open analytic subset of a smooth proper curve.
\end{bem}

We recall the following definitions from \cite[\S 5.6]{CLD}.

\begin{defn} A function $f\colon W\to \R$ on an open subset $W$ of $\Xan$ is called \emph{locally psh-approximable} if every point of $W$ has a neighborhood $V$ in $W$ such that $f$ is the uniform limit of smooth psh functions $f_i$ on $V$. 
A function  $f$ is \emph{locally approximable} if it is locally the difference of two locally psh-approximable functions.	
Furthermore, we say that $f$ is \emph{globally psh-approximable} (resp.~\emph{globally approximable}) on $W$ if $f$ is a uniform limit (resp.~the difference of two uniform limits) of smooth psh functions on $W$.

Let $f\colon W\to \R$ be a locally psh-approximable function, then there exists a unique positive Radon measure $\text{MA}(f)$ on $W$ such that for every open subset $V\subset W$, $g\in C^\infty_c(V)$ and smooth psh functions $f_i\in C^{\infty}(V)$ converging uniformly to $f|_V$, we have
	$$\int_{V} g~d\text{MA}(f)=\lim\limits_{i\to \infty}\int_{V} g\wedge (d'd''f_i)^n,$$
where $n=\dim(X)$ and $(d'd''f_i)^n$ is defined as the $n$-th wedge product of $d'd''f_i$.
We call $\text{MA}(f)$ the \emph{Monge-Amp\`ere measure} of $f$. 

For a locally approximable function $f$ that is given locally by $f_V^+-f_V^-$ for locally psh-approximable functions $f_V^+$ and $f_V^-$ on $V$, we define the Monge-Amp\`ere measure $\text{MA}(f)$ to be the measure obtained by gluing $\text{MA}(f_V^+)- \text{MA}(f_V^-)$. 
Note that the definition is independent of the decompositions.
For details see \cite[Corollaire 5.6.5 \& 5.6.6]{CLD} and \cite[D\'efinition 5.6.7]{CLD}.
\end{defn}

\begin{lem}\label{Lem Curr MA} Let $X$ be a curve, $W$ be an  open subset of $\Xan$ and $f\colon W\to \R$. 
	\begin{enumerate}
		\item  If $f$ is locally approximable, we have $$d'd''[f](g)=\int _W g~d\text{MA}(f)$$ for every $g\in C^\infty_c(W)$.
		\item  If $f$ is locally psh-approximable, then $f$ is in particular psh. 
	\end{enumerate}
\end{lem}

\begin{proof} At first, note that every locally (psh-)approximable function is continuous, and so locally integrable.
We start with assertion i) and assume that $f$ is locally approximable.
We therefore can cover $W$ by open subsets $V_i$ on which $f$ is the difference of uniform limits $$f=\lim_{k\to \infty}f^+_{ik} - \lim_{k\to \infty}f^-_{ik}$$ of smooth psh functions $f^+_{ik}$ and $f^-_{ik}$ on $V_i$.
Choose a partition of unity $(\eta_j)_{j\in J}$ subordinated to this covering $(V_i)_{i\in I}$ (see Proposition \ref{Prop partition}).
We write for simplicity $V_j$ instead of $V_{i(j)}$ and $f_{jk}$ instead of $f_{i(j)k}$.
Then for every $j\in J$, we have $\eta_jg\in C_c^\infty(V_j)$.
Furthermore, $g=\sum_{j\in J}^{}\eta_jg$ on $W$. 
Since $g$ has compact support on $W$, the sum has to be finite.
Hence, 
	\begin{align*}
		d'd''[f](g)
		&= \sum_{j\in J}\int_{V_{j}}f\wedge d'd''(\eta_jg)	\\& = \sum_{j\in J}\left(\lim_{k\to\infty}\int_{V_{j}}f^+_{jk}\wedge d'd''(\eta_jg)-\lim_{k\to\infty}\int_{V_{j}}f^-_{jk}\wedge d'd''(\eta_jg)\right).\\
	\end{align*} 
Using  the theorem of Stokes \cite[Theorem 5.17]{Gubler} twice, we get $$\int_{V_{j}}f^{\pm}_{jk}\wedge d'd''(\eta_jg)= \int_{V_{j}}\eta_jg\wedge d'd''f^\pm_{jk},$$ and so we finally obtain 
\begin{align*}
	d'd''[f](g)	& = \sum_{j\in J}\left(\lim_{k\to\infty}\int_{V_{j}}\eta_jg\wedge d'd''f^+_{jk} -\lim_{k\to\infty}\int_{V_{j}}\eta_jg\wedge d'd''f^-_{jk} \right)\\	
		& = \sum_{j\in J}\left(\int_{V_{j}}\eta_jg ~d\text{MA}(f^+)-\int_{V_{j}}\eta_jg ~d\text{MA}(f^-)\right)\\
		&=  \sum_{j\in J}\int_{V_{j}}\eta_jg ~d\text{MA}(f)
		 = \int_{W}g~ d\text{MA}(f).
\end{align*}
For assertion ii), we assume that $f$ is locally psh-approximable, i.e.~we can cover $W$ by open subsets $(V_i)_{i\in I}$ such that $f$ is the uniform limit $f=\lim_{k\to\infty}f_{ik}$ of smooth psh functions $f_{ik}$ on $V_i$.
As above let $(\eta_j)_{j\in J}$ be a partition of unity subordinated to $(V_i)_{i\in I}$. 
For every non-negative function $g\in C^\infty_c(W)$, the smooth function $\eta_jg$ is also non-negative as $\eta_j\geq 0$ and has compact support on $V_j$.
From the calculations above, we get $$d'd''[f](g) =\sum_{j\in J}\lim_{k\to\infty}\int_{V_{j}}\eta_jg\wedge d'd''f_{jk}\geq 0.$$ 
This proves that $f$ is psh on $W$.
\end{proof}
\begin{lem}\label{Lemma MA(f_i)=d'd''(f_i)}
	Let $X$ be a curve, $f\colon W\to \R$ be a globally psh-approximable function, i.e.~$f$ is the uniform limit of smooth psh functions $f_i$ on $W$, and $g\in C_c^{0}(W)$, then 
	$$\int_W g~ d\text{MA}(f)= \lim\limits_{i\to \infty} \int_W g\wedge d'd''f_i.$$
\end{lem}
\begin{proof}
By \cite[Proposition 3.3.5]{CLD}, there are smooth functions $g_k\in C^\infty_c(W)$ converging uniformly to $g$. 
Then 
\begin{align*}
\int_W g~ d\text{MA}(f)&= \lim\limits_{k\to \infty} \int_W g_k~ d\text{MA}(f)\\
& = \lim\limits_{k\to \infty} \lim\limits_{i\to \infty}\int_W g_k\wedge d'd'' f_i\\
& = \lim\limits_{i\to \infty} \lim\limits_{k\to \infty}\int_W  g_k\wedge d'd'' f_i\\
& = \lim\limits_{i\to \infty} \int_W  g\wedge d'd'' f_i.\\
\end{align*}
Note that we may change the order of the limits since $\int_Wg_k\wedge d'd''f_i$ converges  to $\int_W g\wedge d'd'' f_i$ uniformly in $i\in \N$.
\end{proof}
Model functions and their Monge-Amp\`ere measures play an important role in the proofs of the main theorems in Section \ref{Section Eq}. 
We therefore recall the definition of metrics, formal metrics and model functions.

\begin{defn}
	Let $X$ be a proper and normal variety over $K$ and $L$ be a line bundle on $X$.
	A \emph{continuous metric} $\Vert~\Vert$ on $L^{\an}$ associates to every section $s\in \Gamma(U,L)$ on a Zariski open subset $U$ of $X$ a continuous function $\Vert s\Vert\colon U^{\an}\to [0,\infty)$ such that $\Vert f\cdot s \Vert = |f|\cdot\Vert s \Vert$ holds for every $f\in \OS_X(U)$ and $\Vert s(x)\Vert=0$ if and only if $s$ vanishes in $x$.
	
	We call a continuous metric $\Vert ~\Vert $ on $L^{\an}$ \emph{smooth} (resp.~\emph{psh}) if for every open subset $U$ of $X$ and every invertible section $s$ of $L$ on $U$ the function $-\log\Vert s\Vert$ is smooth (resp.~psh) on $U^{\an}$.
	
	A continuous metric $\Vert~\Vert$ on $L^{\an}$ is \emph{psh-approximable} if there is a sequence of smooth psh metrics $\Vert~\Vert_k$ on $L^{\an}$ such that $\sup_{x\in \Xan} |\log (\Vert s_x(x)\Vert/\Vert s_x(x)\Vert_k)|$  
	  converges to zero for any local section $s_x$ of $L$ that does not vanish in $x$. Clearly, this is independent of the choice of $s_x$.
\end{defn}
\begin{defn}\label{Def model} Let $X$ be a proper variety over $K$ and let $L$ be a line bundle on $X$. 
	A \emph{semistable formal model} of $(X,L)$ is a pair $(\XS,\LS)$ consisting of a semistable  formal model $\XS$ of $\Xan$ and a line bundle $\LS$ on $\XS$ such that $\LS|_{\Xan}\simeq L^{\an}$. 
 Note that we always may assume that $\XS$ is strictly semistable by the semistable reduction Theorem \cite[Ch.~7]{BL} since $K$ is algebraically closed.
 
 Let $(\XS,\LS)$ be a semistable formal model of $(X, L^{\otimes m})$ for $m\in \N_{>0}$.
 Then one can define a continuous metric on $L^{\an}$ in the following way: 
 If $\US$ is a formal trivialization of $\LS$ and  $s$ is a section of $L^{\an}$ on $\US_\eta$ such that $s^{\otimes m}$ corresponds to $\lambda\in \OS_{\Xan}(\US_\eta)$ with respect to this trivialization, then $$-\log\Vert s(x)\Vert_{\LS}:=-\frac{1}{m}\log|\lambda(x) |$$ for all $x\in \US_\eta$.
This definition is independent of all choices and shows immediately that the defined metric is continuous. 
Metrics of this form are called \emph{$\Q$-formal metrics}, and they are called \emph{formal metrics} if $m=1$.

Let $\OS_X$ be the trivial line bundle on $X$.
A function $f\colon \Xan\to \R$ of the form $f=-\log\Vert1\Vert_{\LS}$ for a formal metric $\Vert~\Vert_{\LS}$ associated to a semistable formal model of $(X,\OS_X)$ is called \emph{model function.}
\end{defn}
We have the following statements for model functions due to Chambert-Loir and Ducros and Katz, Rabinoff and Zureick-Brown. 
The first theorem is a direct consequence of a result of Chambert-Loir and Ducros in \cite[\S 6.3]{CLD}.
\begin{Thm} 
\label{Thm Masse Model fct}
Let  $X$ be a projective variety over $K$ and $f=-\log\Vert 1\Vert _{\LS}$ be a model function on $\Xan$ for a semistable formal model $(\XS,\LS)$ of $(X,\OS_X)$.
	\begin{enumerate}
	\item The function $f$ is locally approximable on $\Xan$, and so the Monge-Amp\`ere measure $\text{MA}(f)$ exists. 
	\item We have the following identity of measures 
		$$\text{MA}(f)=\sum_Y \deg_{\LS}(Y)\delta_{\zeta_Y},$$
		where $Y$ runs over all irreducible components of the special fiber  $\XS_s$ and  $\zeta_Y$ is the unique point in $\Xan$ mapped to the generic point of $Y$ under the reduction map (see \cite[Proposition 2.4.4]{BerkovichSpectral}).
	\end{enumerate}
\end{Thm}
\begin{proof} By \cite[Corollaire 6.3.5]{CLD}, we know that there are line bundles $L_1,L_2$ on $X$ with formal models $\LS_1$, $\LS_2$ such that $\OS_X=L_1\otimes L_2^{-1}, \LS=\LS_1\otimes \LS_2^{-1}$, the  corresponding metrics $\Vert~ \Vert _{\LS_{1}}$ and $\Vert ~\Vert _{\LS_{2}}$ are psh-approximable and
	$$-\log \Vert \cdot\Vert _{\LS}=-\log \Vert \cdot\Vert _{\LS_{1}}+\log\Vert \cdot\Vert _{\LS_2}$$
on $\Xan$.
Let $\Vert ~\Vert _{i,k}$ be a sequence of smooth psh metrics  converging uniformly to $\Vert \cdot\Vert _{\LS_{i}}$  on $\Xan$.
For every point $x$ in $\Xan$, let $U$ be an open subset of $X$ with $x\in U^{\an}$ and $s$ an invertible section of $L_1=L_2$ on $U$.
Then $-\log \Vert s\Vert _{\LS_{i,k}}$ is a smooth psh function on $U^{\an}$ for $i=1,2$ and for every $k\in \N$
\begin{align*}
-\log \Vert 1\Vert _{\LS}=-\log \Vert s\Vert _{\LS_{1}}+\log\Vert s\Vert _{\LS_2}=\lim_{k\to \infty} -\log \Vert s\Vert _{1,k}+\lim_{k\to \infty} \log\Vert s\Vert _{2,k}.
\end{align*}
Hence, the function $f=\log\Vert1\Vert_{\LS}$ is locally approximable.
	
The second assertion is shown in
\cite[\S 6.9]{CLD}.
\end{proof}
  The slope formula for line bundles by Katz, Rabinoff, and Zureick-Brown in \cite{KRZB} implies directly the following theorem.
\begin{Thm} 
\label{Thm Masse Model fct 2}
 Let $X$ be a smooth proper algebraic curve over $K$ and $f=-\log\Vert 1\Vert _{\LS}$ be a model function on $\Xan$ for a semistable formal model $(\XS,\LS)$ of $(X,\OS_X)$.
 The restriction $F$ of $f$ to $S(\XS)$ is a piecewise affine function and $f=F\circ \tau_{\XS}$ on $\Xan$.
Furthermore, we have $$\text{MA}(f) = dd^c F=dd^c f.$$
\end{Thm}
\begin{proof}
This follows directly from \cite[Theorem 2.6]{KRZB} using Theorem \ref{Thm Masse Model fct}.
\end{proof}

\section{Comparison of the two notions of subharmonic functions}\label{Section Eq}
In this section, we consider a smooth proper algebraic curve $X$ over $K$. 
We show that every subharmonic function on $\Xan$ is psh and that every continuous psh function on $\Xan$ is subharmonic. 
In particular, the notion of subharmonic and the notion of psh agree for continuous functions.

\subsection{Thuillier's subharmonic functions are also subharmonic in the sense of Chambert-Loir and Ducros} In this subsection, we show that every subharmonic function on an open subset of $\Xan$ is psh. 
First, we prove that every  lisse subharmonic function is psh and that every subharmonic function is locally integrable.
Then we use these results and the fact that a subharmonic function is the  limit of subharmonic lisse functions from Proposition \ref{Prop Net} to prove the general claim. 
\begin{bem} Every signed Radon measure $\mu$ on an open subset $W$ of $\Xan$ defines the following current
$$[\mu]\colon C_c^\infty(W) \to \R, g\mapsto \int_Wg~d\mu$$ (see \cite[Example 6.3]{Gubler}). 
Consider a smooth function $f\in C^\infty(W)$, then we have seen in Proposition \ref{Prop Radon} that the smooth form $d'd''f$ corresponds to a signed Radon measure which we also denote by $d'd''f$. 
Using the theorem of Stokes \cite[Theorem 5.17]{Gubler}, we get $d'd''[f]=[d'd''f]$. Recall that a function is called psh if and only if this current is positive.

Analogously, for every lisse function $f\colon W\to \R$  we get a current $[dd^cf]$ for the corresponding measure $dd^cf$ from Definition \ref{Bem dd^c} which is positive if and only if $f$ is subharmonic (cf. Proposition \ref{Prop lisse subh}).
\end{bem}  
\begin{prop}\label{Lemma currents agree}
	Let $W$ be an open subset of $\Xan$ and $f\colon W\to \R$ a lisse function. For every type  II point $x\in W$, there is an open neighborhood $V$ of $x$ in $W$ on which the currents $d'd''[f]$ and $[dd^cf]$ agree.
\end{prop}
\begin{proof}
	Consider a point $x\in W$ and let $Y$ be a strictly affinoid domain containing $x$ in its interior. 
	Since $f$ is lisse, there is a  strictly semistable formal model $\YS$ of $Y$ with corresponding skeleton $S(\YS)$ and $f=F\circ \tau_{\YS}$ on $S(\YS)$ for a piecewise affine function $F$ on $S(\YS)$.
	
	If $x\notin S(\YS)$, then $f$ is constant on an open neighborhood $V$ of $x$, and so  $$d'd''[f]=[d'd''f]=[0]=[dd^c f]$$ on $V$.
	
	If $x\in S(\YS)$, we may assume that $x$ is a vertex in $S(\YS)$.
	Let $e_1=[x,y_1],\ldots,e_r=[x,y_r]$ be the edges in $S(\YS)$ emanating from $x$, $v_1,\ldots,v_r$ the corresponding tangent directions and $\lambda_i:=d_{v_i}F(x)$.
	By blowing up $\YS$, we may assume that $F|_{e_i}$ is affine, $d(x,y_i)\in \Q$ and that we divide the edge $e_i$ by an additional vertex $y'_i$ with $d(x,y_i)=2d(x,y'_i)$. 
	Denote this blowing up by $\YS'$, and the tangent direction corresponding to $[y'_i,y_i]$ by $v'_i$.
	Define the piecewise affine functions $F_i$ on the closed subset $\Gamma:=\bigcup_{i=1,\ldots,r}[x,y_i] $ of $S(\YS')$ by the following data
	\begin{align*}
F_i(x)=0, 
~d_{v_i}(F_i)(x)=\text{sgn}(\lambda_i)\delta_{ij}\text{ and }
d_{v_i'}(F_i)(y'_i)=-d_{v_i}(F_i)(x).
	\end{align*}
	Set $f_i=F_i\circ \tau_{\YS'}$ on $Y':=\tau_{\YS'}^{-1}(\Gamma)$, which is a strictly affinoid domain in $W$.
	 By definition, $f_i=0$ on $\partial Y'$ for every $i\in\{1,\ldots,r\}$.
	Note that $\partial Y'=\partial \Gamma$ (cf.~\cite[Lemma 5.4]{JW}).
	Hence, we can extend $f_i$ to $\Xan$ by setting $f_i=0$ on $\Xan\backslash Y'$. 
	Then we have a $G$-covering of $\Xan$ on which $f_i$ is piecewise linear, and so $f_i$ is a model function on $\Xan$ \cite[Proposition 8.11]{GK1}.
	Set $\Gamma':=\bigcup_{i=1,\ldots,r}[x,y_i']$ and $V:=\tau_{\YS'}^{-1}((\Gamma')^\circ)$, which is an open neighborhood of $x$ in $W$.
	By the definition of $f_i$ on $V$, we have on $V\subset Y$
\begin{align}\label{Eq Summe Fkt}
	f&= F\circ \tau_{\YS} = F\circ \tau_{\YS'}= (\sum_{i=1}^{r}|\lambda_i|F_i)\circ \tau_{\YS'} + F(x) = \sum_{i=1}^{r}|\lambda_i| \cdot f_i + F(x)
\end{align}
and 
\begin{align}\label{Eq Summe Masse}
	dd^c f= dd^c F= \sum_{i=1}^{r}\lambda_i\delta_x=\sum_{i=1}^{r}|\lambda_i|(dd^c(f_i)).
\end{align}
Since the functions $f_i$ are model functions, we know that they are locally approximable on $V$ and  $\text{MA}(f_i)=dd^c(f_i)$ by Theorem \ref{Thm Masse Model fct} and Theorem \ref{Thm Masse Model fct 2}.
Let $0\leq g\in C_c^\infty(V)$. Then for every $i\in \{1,\ldots,r\}$ we have by Lemma \ref{Lem Curr MA} and $\text{MA}(f_i)=dd^c(f_i)$ that
\begin{align*}
	d'd''[f_i](g)&=  \int_{V} g~d\text{MA}(f_i)= \int_{V} g~dd^c(f_i)= [dd^cf_i](g).
\end{align*}
Linearity and the Equations (\ref{Eq Summe Fkt}) and (\ref{Eq Summe Masse}) imply consequently $d'd''[f](g)=[dd^cf](g)$.
\end{proof}

\begin{prop}\label{Thm lisse}
	Let $W$ be an open subset of $\Xan$ and $f\colon W\to \R$ be a lisse function. If $f$ is subharmonic, then $f$ is psh.  
\end{prop}
\begin{proof}
Note that the property to be psh, i.e.~$d'd''[f]\geq 0$, is a local property by Proposition \ref{Prop psh sheaf}.

If $x$ is of type I or IV, the lisse function $f$ is constant on an open neighborhood $V$ of $x$ in $W$ by the definition of lisse. 
Hence, $d'd''[f]=[d'd''f]=0$ on $V$.

If $x$ is of type II, we have seen in Proposition \ref{Lemma currents agree} that there is an open neighborhood $V$ of $x$ in $W$ such that $d'd''[f]=[dd^cf]$ on $V$.
Since $f$ is subharmonic, the measure $dd^cf$ on $W$ is non-negative by Proposition \ref{Prop lisse subh}, and so we have $d'd''[f]=[dd^cf]\geq0$ on $V$. 
  
  If $x$ is of type III, we choose a strictly affinoid domain $Y$ that contains $x$ in its interior. 
  Since $f$ is lisse, there is a  strictly semistable model $\YS$ of $Y$ with corresponding skeleton $S(\YS)$ and a piecewise affine function $F$ on $S(\YS)$ such that $f=F\circ \tau_{\YS}$ on $Y$.
  
  If $x$ is not contained in $S(\YS)$, then there is as in the first case an open neighborhood $V$ of $x$ in $W$ on which $f$ is constant,
  and so $d'd''[f]=[d'd''f]=0$ on $V$.
  
  If $x\in S(\YS)$, then there is an edge $e$ in $S(\YS)$ such that $x$ lies in the interior of $e$.
  The closed annulus $A:=\tau_{\YS}^{-1}(e)$ is isomorphic to a closed annulus $A'=\trop^{-1}([\val(b),\val(a)])$ in $\G_m^{1,\an}$ for some $a,b\in K^\times$ with $|a|< |b|$ and $\trop:= -\log|T|$.
  Thus, we can identify  $e$ with the real interval $[\val(b),\val(a)]$  via $\trop\circ \Phi$ for a fixed isomorphism $\Phi\colon A\xrightarrow{\sim} A'$.
  Since $e$ is isometric to $[\val(b),\val(a)]$, we can define a function $$\psi\colon [\val(b),\val(a)]\to \R;~ z\mapsto F((\trop\circ \Phi)^{-1}(z)).$$ 
  Then $\psi$ extends to a piecewise affine function on $\trop(\G^{1,\an}_m)=\R$, and its restriction to the connected components of $\R\backslash\{\trop(\Phi(x))\}$ is affine with outgoing slopes at $\trop(\Phi(x))$ equal to the ones of $F$ at $x$ on $e$.
  We have required that $f$ is subharmonic, so the sum of the outgoing slopes at $x$ is greater than or equal to zero by Proposition \ref{Prop lisse subh}. 
  Hence, $\psi$ is convex on $\R$, and so we can find smooth convex functions $\psi_i$ on $\R$ converging uniformly to $\psi$. 
  Then $\psi_i\circ \trop$ are smooth psh functions on $\trop^{-1}(\val(b),\val(a))\subset \G_m^{1,\an}$ by Proposition \ref{Prop smooth psh}.
  By Remark \ref{Pullback Bem}, the pullbacks $f_i:=\Phi^*(\psi_i\circ \trop)$  are smooth psh function on $V':=\tau_{\YS}^{-1}(e^\circ)$ converging uniformly to $f$ on $V'$. 
  Note that by \cite[Lemma 2.13 \& 3.8]{BPR2}, the map $\trop=-\log|T|$ factors through $\tau_A$  and $\tau_A=\tau_{\YS}$ on $A$.
  Hence,  the function $f$ is itself psh on $V'$  by Lemma \ref{Lem Curr MA}, i.e.~$d'd''[f]\geq 0$ on $V'$.  
\end{proof}

\begin{prop}\label{Prop lokal int}
	Let $W$ be an open subset of $\Xan$ and $f\colon W\to [-\infty,\infty)$ be a subharmonic function, then $f$ is locally integrable.
\end{prop}
\begin{proof}
We have to show that $\int_W f\wedge\omega$ is finite for every $\omega\in\AS_c^{1,1}(W)$.
 By \cite[Proposition 5.13]{Gubler}, we may assume that
 the  $(1,1)$-form $\omega$ is of the form $\omega=\trop_U^*\omega_{\trop}$ for a tropical chart $(V,\varphi_U)$ of $W$ and a form $\omega_{\trop}\in \AS^{1,1}_c(\trop_U(V))$.
The closed embedding $\varphi_U$ is given by $\gamma_1,\ldots,\gamma_r\in \OS(U)^\times$, and we denote by $H$ the set of zeros and poles of $\gamma_1,\ldots,\gamma_r$ on $X$. 
We choose a strictly semistable model $\XS$ of $X$ such that the type I points of $H$ lie in distinct connected components of $\Xan\backslash S_0(\XS)$ (see \cite[Theorem 4.11 \& 5.2]{BPR2}), where $S_0(\XS)$ denotes the vertices of the skeleton $S(\XS)$ corresponding to $\XS$.
By \cite[Lemma 2.13 \& 3.8]{BPR2}, we have the following commutative diagram 
\begin{align*}
\begin{xy}
\xymatrix{V\ar[r]^{\trop_U}\ar[d]_{\tau_{\XS}}& \Omega\\
	  S(\XS) \ar[ur]_{\trop_U}&  }
\end{xy}
\end{align*}
where $\tau_{\XS}\colon\Xan\to S(\XS)$ is the retraction map corresponding to $\XS$ and $\Omega:=\trop_U(V)$.
The retraction map $\tau_{\XS}$ is defined in such a way (see \cite[Lemma 3.4 \& Definition 3.7]{BPR2})
 that every connected component of $\Xan\backslash S(\XS)$ is retracted to a single point in $S(\XS)$.
Due to this fact and the commutativity of the diagram, the form
 $\omega=\trop_U^*\omega_{\trop}$ is supported on $S(\XS)$. 
Since the restriction of the subharmonic function $f$ to $S(\XS)$ is continuous by \cite[Proposition 3.4.6]{Th}, we get 
$$\int_{W}f\wedge \omega= \int_{W}(f\circ\tau_{\XS})\wedge\omega$$ 
is finite.
Hence, $f$ is locally integrable.
\end{proof}

\begin{Thm}\label{Korollar Th impl. CLD}
	Let $W$ be an open subset of $\Xan$ and $f\colon W\to [-\infty,\infty)$ be a subharmonic function, then $f$ is psh.
\end{Thm}
\begin{proof} 
	We already know by Proposition \ref{Prop lokal int} that $f$ is locally integrable, so it remains to show that $d'd''[f]$ is non-negative, which is also a local property by Proposition \ref{Prop psh sheaf}.
	Since $\Xan$ is a locally compact Hausdorff space, we can find for every $x\in W$ a relatively compact neighborhood $W'$ of $x$ in $W$. 
	By  Proposition \ref{Prop Net}, there is a decreasing net $\langle f_j\rangle$ of lisse subharmonic functions converging pointwise to $f$ on $W'$.
	Consider a non-negative function $g\in C^\infty_c(W')$. 
	Then there are smooth forms $\omega^+,\omega^-\in \AS^{1,1}(\Xan)$ such that $d'd'' g= \omega^+-\omega^-$ and the corresponding signed Radon measures from Proposition \ref{Prop Radon}  are non-negative \cite[Lemme 5.3.3]{CLD}. 
	By \cite[Corollaire 3.3.4]{CLD}, we can find a smooth non-negative function $\eta\in C^\infty_c(W')$ such that $\eta\equiv 1$ on $\supp(g)$, and hence $$d'd'' g= \eta d'd'' g= \eta \omega^+-\eta\omega^-$$ on $W'$. 
	The smooth $(1,1)$-forms $\eta\omega^{\pm}$ are contained in $\AS_c^{1,1}(W')$ and the corresponding Radon measures are still non-negative and have compact support by Proposition \ref{Prop Radon}.
    Thus, $\int_{W'}f\wedge \eta\omega^{\pm}$ is finite (see Proposition \ref{Prop lokal int}), and we have 
   $$\int_{W'}f\wedge \eta\omega^{\pm}=\lim\limits_{j}\int_{W'}f_j\wedge \eta\omega^{\pm}$$
    by  \cite[Lemma 2.23]{BFJ15}.
    Together, we get
	\begin{align*}
d'd''[f](g)&= \int_{W'}f\wedge \eta d'd''g = \int_{W'}f\wedge \eta\omega^{+} - \int_{W'}f\wedge \eta\omega^{-}\\
&= \lim\limits_{j} (\int_{W'}f_j\wedge \eta\omega^{+} - \int_{W'}f_j\wedge \eta\omega^{-})  = \lim\limits_{j}\int_{W'} f_j\wedge d'd''g. 
	\end{align*}  
	By Proposition \ref{Thm lisse}, we know that  $\int_{W'} f_j\wedge d'd''g \geq 0$ for every $f_j$ in the net, and hence $d'd''[f](g)\geq 0$.
\end{proof}

\subsection{Continuous subharmonic functions in the sense of Chambert-Loir and Ducros are subharmonic in the sense of Thuillier}
In this subsection, we prove that every continuous psh function is subharmonic in the sense of Thuillier. 
The key tool of the proof is the coincidence of the Monge-Amp\`ere measure and the measure $dd^c g$ for model functions (see Theorem \ref{Thm Masse Model fct 2}).

\begin{lem}\label{Lemma Modelfct.}
Let $W$ be an open subset of $\Xan$ and $f\colon W\to \R$ be a continuous function.
If $f$ is not subharmonic, then there is a 
strictly semistable formal model $(\XS,\LS)$ of $(X,\OS_X)$ such that the model function $g:=-\log \Vert1\Vert_{\LS}$  satisfies the following properties: 
\begin{enumerate}
\item $\supp(g)$ is a connected strictly affinoid domain $Y\subset W$;
\item $g$ is strictly positive on $Y\backslash \partial Y$;
\item  $$\int_W f dd^c g< 0.$$ 
\end{enumerate}
\end{lem}
\begin{proof}
    By Corollary \ref{Kor Greenfct}, there is a connected strictly affinoid domain $Y$ and a type II or III point $x$ in $Y\backslash \partial Y$ with  $\int_W f dd^c g^Y_x< 0$.
     Since $g^Y_x$ is lisse, we can find a  strictly semistable formal  model $\YS$ of $Y$ and a piecewise affine function $G^Y_x$ on the corresponding skeleton $S(\YS)$ such that $g^Y_x=G^Y_x\circ \tau_{\YS}$ on $Y$.
	Note that $S(\YS)$ is a closed connected subset of $\Xan$ with the structure of a finite metric graph.
	By blowing up $\YS$, we may assume that there is no edge in $S(\YS)$ whose  endpoints are all contained in the boundary $\partial Y$.
	We will explain in steps, why there is a strictly semistable formal model $(\XS,\LS)$ such that the corresponding model function $g=-\log \Vert 1\Vert _{\LS}$ on $\Xan$ satisfies $g >0$ on $Y\backslash\partial Y$, $g=0$ elsewhere, and $\int_W f dd^c g< 0$.

	\vspace{0.5cm}
	
	\underline{\textbf{0. Step:}} ~Strategy of the proof. 
	
	\vspace{0.3cm}
	
	We construct a $\val(K^\times)$-rational function (see definition below) $G$ on  $S(\YS)$ with the required properties
	\begin{align}\label{Bedingung}
		 G|_{S(\YS)\backslash\partial Y}> 0, ~ G|_{\partial Y}= 0 \text{ and } \int_{ S(\YS)} f dd^c G <0. 
	\end{align}
	Then the function $g:=G\circ \tau_{\YS}$ is piecewise $\Q$-linear in the sense of \cite{GK2} on $Y$.
	Setting $g\equiv 0$ on $\Xan \backslash Y$, we get a well-defined continuous non-negative function $g$ on $\Xan$ and we have a $G$-covering of $\Xan$ on which $g$ is  piecewise $\Q$-linear.
	Thus, there is a semistable formal model $(\XS,\LS)$ such that $g=-\frac{1}{m}\log\Vert 1\Vert _{\LS}$ on $\Xan$ by \cite[Proposition 8.11\& 8.13]{GK1}.
	Note that we always may assume $\XS$ to be strictly semistable  by the semistable reduction Theorem \cite[Ch.~7]{BL} since $K$ is algebraically closed.
	 Then the formal model $(\XS,\LS^{\otimes m})$ of $(X,\OS_X)$ gives the claim.

	Before we start with the construction, note that a $\val(K^\times)$-rational function is a piecewise affine function on  $S(\YS)$ (we refine the vertex set such that $G$ is affine on every edge) such that the following properties are satisfied
	\begin{enumerate}[label=(\alph*)]
		\item $dd^cG$ is only supported on points of type II.
		\item $G$ has values in $\val(K^\times)$ at every vertex of  $S(\YS)$. 
		\item $G$ has rational slopes.
	\end{enumerate}

	 \vspace{0.5cm}
	\begin{tabular}{l l}
		\underline{\textbf{1. Step:}}& Replace $G^Y_x$ by a piecewise affine function $G'$  such that $(\ref{Bedingung})$ is still\\ &  satisfied and  (a) additionally  holds. 
	\end{tabular}
	\vspace{0.3cm}
	
	If $x$ is of type II, the support of $dd^c G^Y_x$  consists only of type II points. 	
	If $x$ is of type III, we use that the points of type II are dense in $\Xan$. 
	Let $e=[y_1,y_2]$ be the edge of  $S(\YS)$ having $x$ in its interior.  
	Let $x_n$ be a sequence of type II points in $e$ converging to $x$ with respect to the skeletal metric. 
	Consider the piecewise affine functions $G_n$ on  $S(\YS)$ that are given by the affine function on $[y_1,x_n]$ (resp.~on $[x_n,y_2]$) connecting the points $G^Y_x(y_1)$ and $G^Y_x(x_n)$ (resp.~$G^Y_x(x_n)$ and $G^Y_x(y_2)$) and $G_n\equiv G^Y_x$ on $S(\YS)\backslash e$. 
	It is easy to see that the slopes converge to the ones of $G^Y_x$, i.e.~$dd^c G_n\to dd^c G^Y_x$ for $n\to \infty$. 
	Furthermore, $f$ is continuous, so we can find $n$ big enough such that $|f(x_n)dd^c G_n (x_n)- f(x)dd^c G^Y_x (x)|$ is so small that we still have $\int_ {S(\YS)} f dd^c G_n <0$ (cf.~\cite[Proposition 3.3.4]{Th}). 
	Set $G':=G_n$ for such an $n$. 
	
	 \vspace{0.5cm}
\begin{tabular}{l l}
	\underline{\textbf{2. Step:}}& Replace $G'$ by a piecewise affine function $G''$  such that $(\ref{Bedingung})$ and (a)\\ & are still satisfied and  (b) additionally holds. 
\end{tabular}
\vspace{0.3cm}

	Due to normalizing the absolute value $|~|$, we assume that $\Q$ is contained, and so dense, in $\val(K^\times)$.
	Let $z$ be a vertex of  $S(\YS)$. If $G'(z)$ is not in $\Q$ and $(a_n)_n$ is a sequence of rational points converging to $G'(z)$, then the slopes of the piecewise affine functions $G'_n$ on  $S(\YS)$ resulting by replacing $G'(z)$ by $a_n$ converge to the slopes of $G'$. 
	Thus, we can choose an $n\in \N$ such that $G'':=G_n'$ still satisfies $(\ref{Bedingung})$ and  takes only values in $\Q$ at every vertex of  $S(\YS)$.
	We choose these values such that $G''$ is still strictly positive on  $S(\YS)\backslash \partial Y$.

	 \vspace{0.5cm}
\begin{tabular}{l l}
	\underline{\textbf{3. Step:}}& Replace $G''$ by a piecewise affine function $G$  such that $(\ref{Bedingung})$, (a) and (b)\\ & are still satisfied and  (c) additionally holds. 
\end{tabular}
\vspace{0.3cm}
	
	Now, consider an edge $e=[y_1,y_2]$ of  $S(\YS)$. 
	Denote by $d(y_1,y_2)$ the distance between these two points with respect to the skeletal metric. 
	If $d(y_1,y_2)\in \Q$, we are done. 
	If not, we can find points $y_1'$ and $y_2'$ of type II in $e$ with  distance $d(y_1',y_2')\in \Q$ arbitrary close to $d(y_1,y_2)$.
	These points are chosen so that we can decompose the edge $e$ into $[y_1,y_1']\cup [y'_1,y'_2]\cup [y'_2,y_2]$.
	Let $G$ be the piecewise affine function on  $e$ defined by the data  $G\equiv G''(y_1)$ on $[y_1,y_1']$, $G\equiv G''(y_2)$ on $[y'_2,y_2]$ and $G|_{[y_1',y_2']}$ is affine (see Figure 1). Since $G(y'_i)=G''(y_i)\in \Q$, the constructed function $G$ on $e$ has rational slopes. 
    Choose $y'_i$ with $d(y_1',y_2')$ close enough to $d(y_1,y_2)$ such that $(\ref{Bedingung})$ is still satisfied.
    Note that this is possible since $f$ is continuous (cf.~\cite[Proposition 3.3.4]{Th}).

	We do this for all edges except the ones containing the boundary points $\partial Y$, where we just move the other vertex. 
	Then the function $G$ on $S(\YS)$ has slopes in $\Q$ and is consequently the required function.	
	
	 \vspace{1 cm}
	
	\begin{tikzpicture}
	\draw[blue, ultra thick,-]  (11,15) node[blue] {\textbullet}node [blue, below = 5]{$y_0$} -- (12,15) node[blue] {\textbullet}node [blue, above = 5]{$\tilde{y}_0$};
	\draw[blue, ultra thick,-]  (12,15) node[blue] {\textbullet}node [blue, below = 5]{} -- (15,15) node[blue] {\textbullet}node [blue, above = 5]{$\tilde{y}_1$};
		\draw[blue, ultra thick,-]  (15,15) node[blue] {\textbullet}node [blue, below = 5]{} -- (16,15) node[blue] {\textbullet}node [blue, below = 5]{$y_1$};
		\draw[blue, ultra thick,-]  (16,15) node[blue] {\textbullet} -- (17,15) node[blue] {\textbullet}node [blue, above = 5]{$y'_1$};	
	\draw[blue, ultra thick,-]  (17,15) node[blue] {\textbullet} -- (20,15) node[blue] {\textbullet}node [blue, above = 5]{$y'_2$};
		\draw[blue, ultra thick,-]  (20,15) node[blue] {\textbullet} -- (21,15) node[blue] {\textbullet}node [blue, below = 5]{$y_2$};
	\draw(9, 14) node[blue] {$~$};
	\draw(18.6, 13.5) node[blue] {$e$};
	\draw(11, 17) -- (12, 17);
	\draw(12, 17) -- (15, 18);
	\draw(15, 18) -- (17, 18);
	\draw(17, 18) -- (20, 16);
		\draw(20, 16) -- (21, 16);
	\draw(21, 16.5) node[] {$G''(y_2)$};
	\draw(16, 18.5) node[] {$G''(y_1)$};
		\draw(16, 18) node[] {$|$};
				\draw(11, 17) node[] {$|$};
		\draw(11, 17.5) node[] {$G''(y_1)$};
					\draw(21, 16) node[] {$|$};
	\draw(13.5, 13.5) node[blue] {$\tilde{e}$};	
	\draw[decorate,decoration={brace, mirror}, yshift=-4ex] (12, 15.3) -- node[below=0.4ex] {$\in \Q$} (15,15.3);	
	\draw[decorate,decoration={brace, mirror}, yshift=-4ex] (17, 15.3) -- node[below=0.4ex] {$\in \Q$} (20,15.3);		
	\end{tikzpicture}
	\begin{center}
		~~~~FIGURE 1
	\end{center}
\end{proof}

\begin{Thm}\label{folgt subh}
 	 	Let $W$ be an open subset of $\Xan$ and $f\colon W\to \R$ be a continuous function. If $f$ is psh, then $f$ is subharmonic.
 \end{Thm}
\begin{proof} 
The proof is by contradiction. 
We assume that there is a strictly semistable formal model  $(\XS,\LS)$ of $(X,\OS_X)$ such that $\int_W f dd^c g< 0$ for  $g:=-\log\Vert 1\Vert_{\LS}$ as in Lemma \ref{Lemma Modelfct.}.

\vspace{0.5cm}
\underline{\textbf{0. Step:}} Strategy of the proof.
\vspace{0.3cm}

We know that $g=G\circ \tau_{\XS}$ on $\Xan$ for the piecewise affine function $G=g|_{S(\XS)}$ on $S(\XS)$, the Monge-Amp\`ere measure exists and $dd^c g=dd ^c G=\mathrm{MA}(g)$ by Theorem \ref{Thm Masse Model fct 2}. 

Assume that we have the following situation: 
We can write $g=g^+-g^-$ on an open subset $V$ of  $W$ that contains the connected strictly affinoid domain $Y=\supp(g)$ such that 
\begin{enumerate}[label=(\alph*)]
\item $g^{\pm}$ is the uniform limit of smooth psh functions $g^{\pm}_{k}$ on $V$, i.e.~$g$ is globally approximable on $V$,
\item $g_k:=g^+_{k}-g^-_{k}\geq 0$ on $V$, and
\item $g_k$ has compact support on $V$.
\end{enumerate}
Moreover, we want to have a connected open subset $V'$ of $V$ with $Y\subset V'$, $\overline{V'}\subset V$ and $g_k\in C_c^\infty(V')$, and  a continuous map $\eta$ on $V$ such that $\eta\equiv 1$ on $V'$ and $\eta$ has compact support in $V$.

In this situation, we have $\eta f\in C_c^0(V)$, $f=\eta f$ on $V'$ and by the definition of the Monge-Amp\`ere measure and Theorem \ref{Thm Masse Model fct 2}
$$dd^cg=\mathrm{MA}(g)=\mathrm{MA}(g^+) - \mathrm{MA}(g^-)$$
 is a measure with finite support contained in $Y\subset V'\subset V$. Due to our assumption, this implies
\begin{align*}
0> \int_{W}f~ d\mathrm{MA}(g) 
=\int_{V}\eta f~ d\mathrm{MA}(g)
=\int_{V}\eta f~ d\mathrm{MA}(g^+)-\int_{V}\eta f~ d\mathrm{MA}(g^-).
\end{align*}
Applying Lemma  \ref{Lemma MA(f_i)=d'd''(f_i)} to the right hand side and using $g_k\in C^\infty_c(V')$, we get 
\begin{align*}
0> \int_{W}f~ d\mathrm{MA}(g) &
=\lim_{k\to\infty}\int_{V}\eta f\wedge d'd''g^+_{k}-\lim_{k\to\infty}\int_{V}\eta f\wedge d'd''g^-_{k}\\&=\lim_{k\to\infty}\int_{V}\eta f\wedge d'd''g_{k}\\
&=\lim_{k\to\infty}\int_{V'}f\wedge d'd''g_{k}\\
& = \lim_{k\to\infty} d'd''[f|_{V'}](g_k).
\end{align*}
We know that $d'd''[f]$ is positive on $W$, and so it is on $V'$. 
Thus, $$\lim_{k\to\infty} d'd''[f|_{V'}](g_k)\geq 0,$$ and so we have a contradiction.
Hence, $f$ has to be subharmonic.

We explain in several steps how to construct $V$, $V'$, $\eta$  and the functions $g^+_{k}$, $g^-_{k}$ such that one has the described situation. 

\vspace{0.5cm}

\underline{\textbf{1. Step:}} ~~Show that the function $g$ is globally approximable on $W$. 

\vspace{0.3cm}

The curve $X$ is projective, so we may assume $\XS$ to be projective as well and can therefore find very ample line bundles $\LS_1,\LS_2$ such that  $\LS=\LS_1\otimes \LS_2^{-1}$. Thus, we can write $$g=\log\Vert 1\Vert_{\LS}=-\log\Vert s_1\Vert_{\LS_1}+\log\Vert s_2\Vert_{\LS_2}$$
on $W$ for global sections $s_1,s_2$ that coincide on the generic fiber.
Since we may work on every open subset of $W$ containing the compact subset $Y$ by Proposition \ref{Prop psh sheaf}, we may assume that $s_1$ and $s_2$ do not vanish on $W$.
Due to the very ampleness, we can find closed embeddings $\varphi_i\colon \XS\to \mathbb{P}_R^{n_i}$ such that $\LS_i\simeq \varphi_i^* \OS_{\mathbb{P}_R^{n_i}}(1)$ and $s_i=\varphi_i^*x_{j_i}$ for $i=1,2$. 
Here, let $x_0,\ldots,x_{n_i}$ be the coordinates of $\mathbb{P}_R^{n_i}$ and $j_i\in \{0,\ldots,n_i\}$. 
Without loss of generality, we assume $x_{j_1}=x_{n_1}$ and $x_{j_2}=x_{n_2}$.
These closed embeddings induce closed embeddings $\varphi_i\colon X\to \PB_K^{n_i}$ into the toric varieties $\PB_K^{n_i}$.
Then 
\begin{align*}
g&=-\log\Vert s_1\Vert_{\LS_1}+\log\Vert s_2\Vert_{\LS_2}\\
&=((\varphi_1^{\an})^*(-\log |x_{n_1}|+\max_{j\in \{0,\ldots,n_1\}}\log|x_j|)-(\varphi_2^{\an})^*(-\log|x_{n_2}|+\max_{j\in \{0,\ldots,n_2\}}\log|x_j|))
\end{align*}
on $W$.
We  approximate the functions $\phi_i:=-\log |x_{n_i}|+\max_{j\in \{0,\ldots,n_i\}}\log|x_j|$ for $i=1,2$ by smooth convex functions on $\{\eta \in \PB_K^{n_i,\an}\mid|x_{n_i}|_\eta\neq 0\}$ as in \cite[Proposition 6.3.2]{CLD}.

For every $k\in \N_{>0}$ and $n\in \N$, there is a smooth, convex function  $M_{n,\frac{1}{k}}$ on $\R^{n+1}$  that is non-decreasing in every variable and has the following properties
\begin{enumerate}
	\item $\max(t_0,\ldots,t_n)\leq M_{n,\frac{1}{k}}(t_0,\ldots,t_n) \leq \max(t_0,\ldots,t_n) +\frac{1}{k}$;
	\item If $t_l+\frac{2}{k}\leq \max_{j\neq l} t_j$, then $M_{n,\frac{1}{k}}(t_0,\ldots,t_n)=M_{n-1,\frac{1}{k}}(t_0,\ldots,\widehat{t_l},\ldots,t_n)$;
	\item For all $t\in \R$, we have $M_{n,\frac{1}{k}}(t_0+t,\ldots,t_n+t)=M_{n,\frac{1}{k}}(t_0,\ldots,t_n)+t$.
\end{enumerate}

Define for $k\in \N_{>0}$ and $i\in \{1,2\}$ the following function
\begin{align*}
\phi_{i,k}&:=-\log |x_{n_i}|+M_{n_i,\frac{1}{k}}(\log|x_0|,\ldots,\log|x_{n_i}|)
\end{align*}
 on $\PB_K^{n_i,\an}$.
 For every point $\zeta$ in $\{\eta \in \PB_K^{n_i,\an}\mid|x_{n_i}|_\eta\neq 0\}$ there is an open neighborhood of  $\zeta$ such that 
 $\phi_{i,k}$ is smooth and psh on this neighborhood.
 Both are local properties and $s_i$ does not vanish on $W$, so  $(\varphi_i^{\an})^*\phi_{i,k}$ is a smooth psh function on $W$ (see Remark \ref{Pullback Bem})
 converging uniformly to $-\log\Vert s_i\Vert_{\LS_i}$ on $W$.
We use in the following the notations $g^+:= -\log\Vert s_1\Vert_{\LS_1}$ and $g^-:= -\log\Vert s_2\Vert_{\LS_2}$.
 Furthermore, we set $g^+_{k}:=(\varphi_1^{\an})^*\phi_{1,k}$ and $g^-_{k}:=(\varphi_2^{\an})^*\phi_{2,k}$. 
Note that we have by construction  \begin{align*}
	g^{\pm}\leq g_k^\pm\leq g_k^\pm +1/k.
\end{align*}
 
 \vspace{0.2cm}

 These functions do not necessarily satisfy (b) and (c), so  we need to modify $g^{\pm}_{k}$. 
 
 \vspace{0.5cm}
 \begin{tabular}{l l}
 \underline{\textbf{2. Step:}}& Construct a suitable $V$ and study the behavior of  $g^{\pm}_{k}$ outside of $Y$:\\&
 \end{tabular}
 \vspace{0.3cm}
 
 At the end of Step 2, one can find an illustration of the construction in Figure 2.
 The boundary of the strictly affinoid domain $Y$ is a finite set of points of type II.  
 By blowing up our model $\XS$, we may assume that the points $\partial Y$  are vertices in $S(\XS)$. 
 Note that we always may assume that an admissible blowing up is strictly semistable again by the semistable reduction Theorem \cite[Ch.~7]{BL2}. 
 Consider a point $y\in \partial Y$. Since $y$ is of type II, \cite[Corollary 4.27 \& Theorem 4.11]{BPR2} tells us that there are a strictly semistable formal model $\XS_y$ of $X$ and a star-shaped open neighborhood $\Omega_y$ of $y$ in $S(\XS_y)$ such that $\tau_{\XS_y}^{-1}(\Omega_y)$ is an open neighborhood of $y$ in $W$.
 Here, a star-shaped open neighborhood $\Omega_y$ of $y$ in $S(\XS_y)$ is a simply-connected open neighborhood of $y$ in $S(\XS_y)$ such that the intersection of $\Omega_y$ with any edge $e$ in $S(\XS_y)$ emanating from $y$ is a half-open interval $I_{y,e}=[y,x_e)$ with 
 endpoints $y$ and $x_e$ of type II. 
 By blowing up $\XS$ and modifying $\Omega_y$, we may assume that we can find this star-shaped open neighborhood $\Omega_y$ in $S(\XS)$. We explain how to do this. 
 The model $\XS$ has to be blown up such that every vertex of $S(\XS_y)$ is a vertex in the new skeleton $S(\XS)$. Then we can modify $\Omega_y$ in the following way.
 Consider an edge $e$ of the new skeleton $S(\XS)$. Then the interior of $e$ is either contained in an edge $\tilde{e}$ of $S(\XS_y)$  or lies in a connected component of $\Xan\backslash S(\XS_y)$ isomorphic to an open ball. 
 In the first case, we shrink $I_{y,\tilde{e}}$ so that $I_{y,e}$ is a half-open interval in $e$. 
 Note that $\tau_{\XS}^{-1}(I_{y,e})=\tau_{\XS_y}^{-1}(I_{y,e})$, and so it is still contained  in $W$. 
 In the second case, we just add a new half-open interval $I_{y,e}$ to $\Omega_y$.
 Then $\tau_{\XS}^{-1}(I_{y,e})\subset \tau_{\XS_y}^{-1}(y)\subset W$.
 We do this blowing ups and modifications for all boundary points.
 Moreover, we always choose $\Omega_y$ such that $\overline{\Omega_y}\subset W$.
Before we can construct $V$, we have to blow up $\XS$ one more time. 
We find this admissible formal blowing up  $\XS'$ of $\XS$ such that $Y=\YS_\eta$ for a formal open subset $\YS$ of $\XS'$, which is possible by \cite[Lemma 4.4]{BL}. Then we have $Y=\tau_{\XS'}^{-1}(S(\XS')\cap Y)$. 
As described above, we modify $\Omega_y$ such that it is a star-shaped open neighborhood of $y$ in $S(\XS')$.
Altogether, we have a strictly semistable formal model $\XS'$ of $X$ such that every boundary point $y$ of $Y$ has an open neighborhood $\tau_{\XS'}^{-1}(\Omega_y)$ in $W$ for an open star-shaped neighborhood $\Omega_y$ of $y$ in $S(\XS')$ and $Y=\tau_{\XS'}^{-1}(S(\XS')\cap Y)$.

 For every point $z\in (S(\XS')\cap Y)\backslash \partial Y$, we choose some star-shaped open neighborhood $\Omega_z$ of $z$ in $S(\XS')$ with $\Omega_z\subset Y\backslash \partial Y$.
 Then $\tau^{-1}_{\XS'}(\Omega_z)$ is automatically contained in $W$ due to $Y=\tau_{\XS'}^{-1}(S(\XS')\cap Y)$ and $Y\subset W$.
 
We have constructed for every point in $S(\XS')\cap Y$ an open neighborhood of it in $S(\XS')$, $\Omega_y$ for $y\in \partial Y$ and $\Omega_z$ for $z\notin \partial Y$, and so these open subsets clearly cover our compact subset $S(\XS')\cap Y$.
Thus, there is a finite subset $Y_0$ of $S(\XS')\cap Y$  such that $$Y= \tau_{\XS'}^{-1}(S(\XS')\cap Y)\subset \bigcup_{z\in Y_0}\tau_{\XS'}^{-1}(\Omega_z).$$
By construction, the set of boundary points $\partial Y$ is contained in $Y_0$.
Furthermore, we choose $Y_0$ minimal, i.e.~removing one open subset $\tau_{\XS'}^{-1}(\Omega_z)$ from the covering would no longer cover $Y$.
Set $$V:=\bigcup_{z\in Y_0}\tau_{\XS'}^{-1}(\Omega_z),$$ then $V$ is an open subset of $W$ containing $Y=\supp(g)$.
 
 Let $y$ be a point in $\partial Y$, $e$ an edge emanating from $y$ in $ S(\XS')$ not contained in $Y$ and $I_{y,e}$ the corresponding half-open interval in the star-shaped open neighborhood $\Omega_y$.
 Note that $g_{|_e}=0$. 
 We may shrink   the half-open interval $I_{y,e}$ in $e$, and so $\Omega_y$ and $V$, such that
 \begin{align*}
 g^+&=	(\varphi_1^{\an})^*(-\log |x_{n_1}|+\max_{j\in \{0,\ldots,n_1\}}\log|x_j|)=(\varphi_1^{\an})^*(-\log |x_{n_1}|+\log|x_{l_1}|)	\\
 g^-&=	(\varphi_2^{\an})^*(-\log|x_{n_2}|+\max_{j\in \{0,\ldots,n_2\}}\log|x_j|=  (\varphi_2^{\an})^*(-\log|x_{n_2}|+\log|x_{l_2}|))
 \end{align*}
 on $I_{y,e}$ for some $l_i\in \{0,\ldots,n_i\}$. 
 Define the map $N_{i,\max}\colon I_{y,e}\to \N$ as follows $$N_{i,\max}(\eta):=| \{j\in\{0,\ldots,n_i\}| (\varphi_i^{\an})^*(\log |x_{j}|)=(\varphi_i^{\an})^*(\log |x_{l_i}|) \}|.$$
 Shrink $I_{y,e}$ again such that $(\varphi_i^{\an})^*(\log |x_{j}|)$ is affine on $I_{y,e}$ for every $j\in\{0,\ldots,n_i\}$.
Consequently, the function $N_{i,\max}$ is constant on $I_{y,e}\backslash\{y\}$ as $(\varphi_i^{\an})^*(\log |x_{l_i}|)= \max_{} (\varphi_i^{\an})^*(\log |x_{j}|)$ on $I_{y,e}$.
 Hence, we write $N_{i,\max}$ for this constant value.
 
 Define  
 \begin{align*}
 	\Omega_{i,k}:=&\{\eta\in I_{y,e}\mid  (\varphi_i^{\an})^*(\log |x_{j}|)(\eta)+\frac{2}{k}<(\varphi_i^{\an})^*(\log |x_{l_i}|)(\eta)\\
 	&~~~ \text{if } (\varphi_i^{\an})^*(\log |x_{j}|)|_{I_{y,e}}\neq (\varphi_i^{\an})^*(\log |x_{l_i}|)|_{I_{y,e}},~j\in\{0,\ldots,n_i\} \}	
 \end{align*}
 Let $N'\in \N$ such that $\Omega_k:=(\Omega_{1,k}\cap \Omega_{2,k})\backslash \{y\}$  is a non-empty connected open subset of $I_{y,e}$ for every $k\geq N'$. 
 We have $\Omega_{k}\subset \Omega_{k+1}$, and we work in the following with $\Omega_{y,e}:=\Omega_{N'}$.
 Then we get  for every $k\geq N'$ using properties i)-iii) that
 \begin{align}\label{Equ i=1}
 	g^+_{k}&= (\varphi_1^{\an})^*(-\log|x_{n_1}|)+M_{n_1,\frac{1}{k}}((\varphi_1^{\an})^*(\log|x_{0}|),\ldots,(\varphi_1^{\an})^*(\log|x_{n_1}|))\\
 	&=(\varphi_1^{\an})^*(-\log|x_{n_1}|)+M_{N_{1,\max}-1,\frac{1}{k}}((\varphi_1^{\an})^*(\log|x_{l_1}|),\ldots,(\varphi_1^{\an})^*(\log|x_{l_1}|))\nonumber \\
 	&= (\varphi_1^{\an})^*(-\log|x_{n_1}|)+(\varphi_1^{\an})^*(\log|x_{l_1}|)+M_{N_{1,\max}-1,\frac{1}{k}}(0,\ldots,0) \nonumber \\
 	&= g^+ + M_{N_{1,\max}-1,\frac{1}{k}}(0,\ldots,0) \nonumber 
 \end{align}	on $\Omega_{y,e}$. Set $C^+_{k}:=M_{N_{1,\max}-1,\frac{1}{k}}(0,\ldots,0)\in [0,\frac{1}{k}]$.
  Analogously, 
  \begin{align}\label{Equ i=2}
  g^-_k = g^- + M_{N_{2,\max}-1,\frac{1}{k}}(0,\ldots,0)
  \end{align}
on $\Omega_{y,e}$. Set $C^-_{k}:=M_{N_{2,\max}-1,\frac{1}{k}}(0,\ldots,0)\in [0,\frac{1}{k}]$.
 Due to  $g=g^+-g^-=0$ on $\Omega_{y,e}$, we have
 \begin{align*}	
 g^+_{k}-g^-_{k}	=  M_{N_{1,\max}-1,\frac{1}{k}}(0,\ldots,0)  - M_{N_{2,\max}-1,\frac{1}{k}}(0,\ldots,0) 
 \end{align*}
on $\Omega_{y,e}$.
 \vspace{1 cm}
 
\begin{tikzpicture}
\draw[cyan, ultra thick,-]  (12,15)node[] {} -- (18.625,15);
\draw[blue, ultra thick,-]  (9,15) node[blue] {\textbullet} -- (12,15) node[blue] {\textbullet}node [blue, below = 15]{$y$};
\draw[blue] ([shift={(12,15)}]165:1) arc[radius=1, start angle=165, end angle= 195];
\draw[cyan] ([shift={(17.6,15)}]-15:1) arc[radius=1, start angle=-15, end angle= 15];
\draw[dotted] (19,12.5) -- (8,12.5) -- (8,15.8) -- (19,15.8) -- (19,12.5);
\draw(8.5, 13) node[] {$W$};
\draw(10, 14) node[blue] {$S(\XS')\cap Y$};
\draw(15.5, 13) node[cyan] {$I_{y,e}$};
\draw(18.6, 14.2) node[cyan] {$x_e$};
\draw(12,15) node[cyan, ultra thick]{\LARGE{\textbf{[}}};
\draw[dashed] (12, 19) -- (12, 15);
\draw[dashed] (18.6, 19) -- (18.6, 15.2);
\draw(12, 16.3) -- (18.6, 17.8)node[above, right] {$g^+=g^-$};
\draw (15.5,13.8) node[red]{$\Omega_{y,e}$};
\draw[red] ([shift={(15.3,15)}]-10:2.5) arc[radius=2.5, start angle=-10, end angle= 10];
\draw[red] ([shift={(15.3,15)}]170:2.5) arc[radius=2.5, start angle=170, end angle= 190];
\draw[red, dashed] (12.8, 15) -- (17.8, 15);
\draw[dashed] (12.8, 19) -- (12.8, 15);
\draw[dashed] (17.8, 19) -- (17.8, 15);
\draw(12.8, 16.9) -- (17.8, 18.05)node[above, right] {$g^-_{k}$};
\draw(12.8, 17.2) -- (17.8, 18.35);
\draw(18.2, 18.5) node[] {$g^+_{k}$};
\end{tikzpicture}
\begin{center}
~~~~FIGURE 2
\end{center}
\vspace{0.5cm}
\begin{tabular}{l l}
	\underline{\textbf{3. Step:}}& Modify the constructed functions $g^{\pm}_{k}$ on the open subset $V\subset W$ \\ & such that (b) and (c) are satisfied. 
\end{tabular}
\vspace{0.3cm}

We start with (b).
To ensure that the difference $g^+_{k}-g^-_{k}=(\varphi_1^{\an})^*\phi_{1,k}-(\varphi_2^{\an})^*\phi_{2,k}$ is non-negative on $V$, we work with $g^+_{k}:=(\varphi_1^{\an})^*\phi_{1,k}+\frac{1}{k}$ on $V$ instead of $(\varphi_1^{\an})^*\phi_{1,k}$, which is still a smooth psh function on $V$ converging uniformly to the function $g^+=-\log\Vert s_1\Vert_{\LS_1}$, and 
\begin{align*}
(\varphi_1^{\an})^*\phi_{1,k}+\frac{1}{k}-(\varphi_2^{\an})^*\phi_{2,k}&\geq  (\varphi_1^{\an})^*\phi_{1,k}+\frac{1}{k}-g^++g^--(\varphi_2^{\an})^*\phi_{2,k}\\
&\geq  0 + \frac{1}{k}-\frac{1}{k}=0
\end{align*}
on $V$. 
Note that we have used $g=g^+- g^-\geq 0$ on $V$ and property i)  of $M_{n_i,k}$.

Next, we deal with (c), i.e.~we modify $g^+_{k}$ and $g^-_{k}$ such that $g_k:=g^+_{k}-g^-_{k}$ has compact support on $V$. 
As in Step 2, let $y$ be a boundary point of $Y$, $\Omega_{y}$ a star-shaped open neighborhood of $y$ and $I_{y,e}=[y,x_e)$ be a half-open interval of $\Omega_{y}$ contained in an edge $e$ of $S(\XS')$ emanating from $y$ and pointing outwards of $Y$.
Recall that we have constructed in the second step the open subset $\Omega_{y,e}$ of $I_{y,e}$.
Here, we start with the construction of an affine function on the closed annulus $A:=\tau_{\XS'}^{-1}([y,x_e])$ whose graph intersects the graphs of all our functions $g^{\pm}_{k}$ on $\Omega_{y,e}$ for $k$ big enough.
The annulus $A$ is isomorphic to a closed annulus  $A':=S(a,b)=\trop^{-1}([\val(b),\val(a)])$ in $\G_m^{1,\an}$ for some $a,b\in K^\times$ with $|a|< |b|$ and $\trop:= -\log|T|$.
Let $\Phi_{y,e}\colon A\xrightarrow{\sim} A'$ be an isomorphism. 
For simplicity, we may assume that $\val(b)=0$.
Then we can identify  $I_{y,e}=[y,x_e)$ with the real half-open interval $[0,\val(a))$ via $\trop_{y,e}:=-\log|T|\circ\Phi_{y,e}$.
Choose points $\zeta_{y,e},\zeta'_{y,e}\in \Omega_{y,e}\subset e$, and $m\in \N$, $c\in \R$ such that the function $$\psi_{y,e}\colon I_{y,e}\to \R;~\zeta\mapsto m\cdot \trop_{y,e}(\zeta)+c$$ satisfies 
$$\psi_{y,e}(\zeta_{y,e})=((\varphi_1^{\an})^*(-\log|x_{n_1}|)+(\varphi_1^{\an})^*(-\log|x_{l_1}|))(\zeta_{y,e})=g^+(\zeta_{y,e}),$$  and $\psi_{y,e}(\zeta'_{y,e})=g^+_{N'}(\zeta'_{y,e})$.
Recall that $N'$ was fixed in Step 2 to define $\Omega_{y,e}$.
Since $g^{\pm}_{k}$ converges uniformly to $g^+=g^-$ on $\Omega_{y,e}$, there  is an $N''\geq N'$ such that 
$$\sup_{x\in \Omega_{y,e}}(g^{\pm}_{k}-g^{\pm})=\sup_{x\in \Omega_{y,e}}|g^{\pm}_{k}-g^{\pm}|\leq C^+_{N'}$$ for every $k\geq N''.$
By (\ref{Equ i=1}), we have $g^+_{N'}-g^{\pm}= C^+_{N'}+1/N'$, and hence
 $$g^{\pm}_{k} - g^+_{N'}= g^{\pm}_{k} - g^{\pm}+g^{\pm}- g^+_{N'}<0$$ on $\Omega_{y,e}$ for every $k\geq N''$.
Thus, for every $k\geq N''$ there is a point $\zeta_k^+$ in $\Omega_{y,e}$ such that $\psi_{y,e}(\zeta_k^+)=g^+_{k}(\zeta_k^+)$. 
Due to $g^+= g^-\leq g^-_{k}\leq g^{+}_k$ by (\ref{Equ i=2}) and construction, there is also for every $k\geq N''$ a point $\zeta_k^-$ in $\Omega_{y,e}$ such that  $\psi_{y,e}(\zeta_k^-)=g^-_{k}(\zeta_k^-)$.
Recall that  $g^-=g^+$ and $g^+_k$ and $g^-_k$ with $k\geq N'$ are affine on $\Omega_{y,e}$. 

We choose $\varepsilon$ such that 
\begin{align*}
\Omega_{y,e,<}&:=\{\zeta\in \Omega_{y,e}\mid \psi_{y,e}(\zeta)+2\varepsilon< g^+(\zeta)\}\\
\Omega_{y,e,>}&:=\{\zeta\in \Omega_{y,e}\mid \psi_{y,e}(\zeta)> g^+(\zeta)+2\varepsilon\}
\end{align*} are non-empty open subsets of $\Omega_{y,e}$, and we set $\Gamma_{y,e}:=\Omega_{y,e}\backslash(\Omega_{y,e,<}\cup \Omega_{y,e,>})$.

In the following, we smoothen the piecewise affine functions $\max(g^{\pm}_{k},\psi_{y,e})$ in a proper way.
One can construct a smooth symmetric convex $1$-Lipschitz continuous function $\theta_{\varepsilon}\colon\R\to (0,\infty)$ such that $\theta_{\varepsilon}(a)=|a|$ if $|a|\geq \varepsilon$.
We set 
\begin{align}\label{m_eps}
m_\varepsilon(a,b):=
\frac{a+b+\theta_{\varepsilon}(a-b)}{2}.
\end{align}

Then the smooth function $m_\varepsilon\colon\R^2\to \R$ satisfies the following properties:
\begin{enumerate}
\item $m_\varepsilon$ is convex;
\item $\max(a,b)\leq m_\varepsilon(a,b)\leq \max(a,b) + \frac{\varepsilon}{2}$;
\item $m_\varepsilon(a,b)=\max(a,b)$ whenever $|a-b|\geq \varepsilon$; 
\item $m_\varepsilon$ is increasing in every variable.
\end{enumerate} 
We define the functions 
\begin{align*}
\widetilde{g}^+_{k}:=m_{\varepsilon}(g^+_{k},\psi_{y,e}),\\
\widetilde{g}^-_{k}:=m_{\varepsilon}(g^-_{k},\psi_{y,e})
\end{align*}
on $\Omega_{y,e}$.
Then $\widetilde{g}^+_{k}$ (resp.~$\widetilde{g}^-_{k}$) coincides  with $g^+_{k}$ (resp.~with $g^-_{k}$) on $\Omega_{y,e,<}$ for every $k\geq N''$ since $g^+= g^-\leq g^-_{k}\leq g^+_{k}$ on $I_{y,e}$ by (\ref{Equ i=2}) and by construction.
The functions $g^+_{k}$ converge uniformly to $g^-=g^+$, so we can choose $N_{y,e}\geq N''$ such that for all $k\geq N_{y,e}$, we have $g^++\varepsilon\geq g^+_{k}\geq g^-_{k}$  on $\Omega_{y,e}$.
Then $\widetilde{g}^+_{k}$ and  $\widetilde{g}^-_{k}$ coincide with $\psi_{y,e}$  on $\Omega_{y,e,>}$ for all $k\geq N_{y,e}$.
 Thus, $\widetilde{g}^+_{k}-\widetilde{g}^-_{k}=0$ on $\Omega_{y,e,>}$ for every $k\geq N_{y,e}$.
We do this for every $y\in \partial Y$ and  for every edge $e$ in $S(\XS')$ emanating from $y$ and  pointing outwards of $Y$. 

\begin{tikzpicture}[decoration=brace]
\draw[cyan, ultra thick,-]  (12,15)node[] {} -- (18.625,15);
\draw[blue, ultra thick,-]  (9,15) node[blue] {\textbullet} -- (12,15) node[blue] {\textbullet}node [blue, below = 15]{$y$};
\draw[blue] ([shift={(12,15)}]165:1) arc[radius=1, start angle=165, end angle= 195];
\draw[cyan] ([shift={(17.6,15)}]-15:1) arc[radius=1, start angle=-15, end angle= 15];
\draw[dotted] (19,12.5) -- (8,12.5) -- (8,15.8) -- (19,15.8) -- (19,12.5);
\draw(8.5, 13) node[] {$W$};
\draw(10, 14) node[blue] {$S(\XS')\cap Y$};
\draw(15.5, 12.8) node[cyan] {$I_{y,e}$};
\draw(18.6, 14.2) node[cyan] {$x_e$};
\draw(12,15) node[cyan, ultra thick]{\LARGE{\textbf{[}}};
\draw[dashed] (12, 19) -- (12, 15);
\draw[dashed] (18.6, 19) -- (18.6, 15.2);
\draw(12, 16.3) -- (18.6, 17.8)node[above, right] {$g^+=g^-$};
\draw[black, dotted](12, 17.7)node[above, left] {$g^++\varepsilon$} -- (18.6, 19.2);
\draw (15.5,13.5) node[red]{$\Omega_{y,e}$};
\draw[red] ([shift={(15.3,15)}]-10:2.5) arc[radius=2.5, start angle=-10, end angle= 10];
\draw[red] ([shift={(15.3,15)}]170:2.5) arc[radius=2.5, start angle=170, end angle= 190];
\draw[red, dashed] (12.8, 15) -- (17.8, 15);

\draw[decorate,decoration={brace, mirror}, yshift=-4ex] (12.8, 15) -- node[below=0.4ex] {$\Omega_{y,e,<}$} (14.2,15);
\draw[decorate,decoration={brace, mirror}, yshift=-4ex] (16.4, 15) -- node[below=0.4ex] {$\Omega_{y,e,>}$} (17.8,15);  
\draw[dashed] (12.8, 19) -- (12.8, 15);
\draw[dashed] (17.8, 19) -- (17.8, 15);
\draw(12.8, 16.9) -- (17.8, 18.05);
\draw(18.2, 17.95) node[] {$g^-_{k}$};
\draw(12.8, 17.2) -- (17.8, 18.35)node[above, right] {$g^+_{k}$};
\draw(12.8, 17.5) -- (17.8, 18.65);
\draw(18.2, 18.83) node[]{$g^+_{N'}$} ;
\draw(12, 13.1) -- (18.6, 21)node[above, right] {$\psi_{y,e}$};
\draw(15.3, 15) node[black]{\textbullet};
\draw(15.3, 15.5) node[black]{$\zeta_{y,e}$};
\draw[black] (14.2, 15) -- (16.4, 15);
\draw(14.2,15) node[black, ultra thick]{\LARGE{\textbf{[}}};
\draw(16.4,15) node[black, ultra thick]{\LARGE{\textbf{]}}};
\draw (15.5,14.5) node[black]{$\Gamma_{y,e}$};
\draw[red](12.8, 16.9) -- (15.72, 17.57);
\draw (10.5, 17) node[red]{$\max(g^\pm_{k}, \psi_{y,e})$};

\draw[red](12.8, 17.2) -- (16.06 , 17.945);
\draw[red](15.72, 17.57)  -- (18.6, 21);
\draw[red](16.06 , 17.945) -- (18.6, 21);
\end{tikzpicture}
\begin{center}
	~~~~~	FIGURE 3
\end{center}

Recall that we work on the open subset $V=\bigcup_{z\in Y_0}\tau_{\XS'}^{-1}(\Omega_z)$ of $W$ containing $Y$, where $\Omega_z$ is a star-shaped open neighborhood of $z$ in $S(\XS')$ and $I_{z,e}=\Omega_z\cap e$ for every edge $e$ in $ S(\XS')$ emanating from $z$. 
We have $\partial Y\subset Y_0$ and the sets $\Omega_y\backslash Y$ with $y\in \partial Y$ are disjoint of all other star-shaped open neighborhoods $\Omega_z$.
We write $S_y$ for the set of edges $e$ in $S(\XS')$ emanating from $y$ and pointing outwards of $Y$.
Then for every $y\in \partial Y$ and  $e\in S_y$, we have constructed the open subset $\Omega_{y,e}$ of $I_{y,e}=[y,x_e)\subset \Omega_y$.

Let $V_Y$ be the connected component of $V\backslash (\bigcup _{y\in \partial Y; e\in S_y} \Gamma_{y,e})$  containing $Y$ and $V_e$ be the connected component of $V\backslash \Gamma_{y,e}$ containing $x_e$.
Note that by the construction of $V$ all connected components $V_e$ are pairwise disjoint.
We can extend $\widetilde{g}^\pm_{k}$ to a continuous function on $V$  by 
\begin{align*}
\widetilde{g}^\pm_{k}:=\begin{cases}
m_{\varepsilon}(g^\pm_{k},\psi_{y,e}\circ \tau_{\XS'}) & \text{ on } \tau_{\XS'}^{-1}(\Omega_{y,e}),\\
g^\pm_{k} & \text{ on } V_Y,\\
\psi_{y,e}\circ \tau_{\XS'} & \text{ on } V_e.
\end{cases}	
\end{align*}  

 Then the functions $\widetilde{g}^+_{k}-\widetilde{g}^-_{k}$ are non-negative and continuous on $V$ with compact support for every $k\geq \max_{y,e} N_{y,e}$. 
 Recall that by property iv), $m_{\varepsilon}$ is increasing in every variable.
 From now on we only consider $k\in \N$ with $k\geq \max_{y\in \partial Y,e\in S_y} N_{y,e}$.

\vspace{0.5cm}
\begin{tabular}{l l}
\underline{\textbf{4. Step:}} & Show that the modified functions $\widetilde{g}^+_{k}$ (resp.~$\widetilde{g}^-_{k}$) converge uniformly to a function\\ & $\widetilde{g}^+$  
 (resp.~$\widetilde{g}^-$) such that $g=\widetilde{g}^+-\widetilde{g}^-$ on $V$.
\end{tabular}  
 
\vspace{0.3cm} 
 
 Define the following functions on $V$ 
\begin{align*}
\widetilde{g}^\pm:=\begin{cases}
m_{\varepsilon}(g^\pm,\psi_{y,e}\circ \tau_{\XS'}) & \text{ on } \tau_{\XS'}^{-1}(\Omega_{y,e}),\\
g^\pm & \text{ on } V_Y,\\
\psi_{y,e}\circ \tau_{\XS'} & \text{ on } V_e.
\end{cases}	
\end{align*} 
 
By construction, these functions are well-defined and continuous.
Since  $g=g^+-g^-$ on $V$ with $g=g^+-g^-=0$ on $V\backslash Y$, we clearly have $g=\widetilde{g}^+-\widetilde{g}^-$ on $V$.

Next, we show that $\widetilde{g}^+_{k} $ (resp.~$\widetilde{g}^-_{k}$) converge uniformly to $\widetilde{g}^+$ (resp.~to $\widetilde{g}^-$) on $V$.
We know that $g^+_{k} $ (resp.~$g^-_{k}$)  converge uniformly to $g^+$ (resp.~to $g^-$) on $V$,
thus  $\widetilde{g}^+_{k}$ (resp.~$\widetilde{g}^-_{k}$) converge uniformly to $g^+$ (resp.~to $g^-$) on $V_Y$.
Since $\widetilde{g}^\pm_{k}=\psi_{y,e}\circ \tau_{\XS'}=\widetilde{g}^\pm$ on $V_e$, it remains to consider the open subset $\tau^{-1}_{\XS'}(\Omega_{y,e})$.
For every $x\in\tau^{-1}_{\XS'}(\Omega_{y,e})$, we have 
\begin{align*}
|\widetilde{g}^\pm_{k}(x)-\widetilde{g}^\pm(x)|&= \left|m_{\varepsilon}(g^\pm_{k}(x),\psi_{y,e}(\tau_{\XS'}(x))-m_{\varepsilon}(g^\pm(x),\psi_{y,e}(\tau_{\XS'}(x)))\right|\\& =
\left|\frac{g^\pm_{k}(x)-g^\pm(x)+\theta_{\varepsilon}(g^\pm_{k}(x)-\psi_{y,e}(\tau_{\XS'}(x)))-\theta_{\varepsilon}(g^\pm(x)-\psi_{y,e}(\tau_{\XS'}(x)))}{2}\right|\\&\leq \left |\frac{g^\pm_{k}(x)-g^\pm(x)}{2}\right|+\left|\frac{g^\pm_{k}(x)-g^\pm(x)}{2}\right | \\&\leq  |g^\pm_{k}(x)-g^\pm(x)| \end{align*}
where we used that $\theta_{\varepsilon}$ is $1$-Lipschitz continuous to get the inequality. 
Due to the uniform convergence of $g^\pm_{k}$ to $g^\pm$ on $V$, which contains $\tau^{-1}_{\XS'}(\Omega_{y,e})$, we get
\begin{align*}
\lim_{k\to \infty}\sup_{x\in \tau^{-1}_{\XS'}(\Omega_{y,e})}|\widetilde{g}^\pm_{k}(x)-\widetilde{g}^\pm(x)|=0.
\end{align*}
\vspace{0.5cm}
\begin{tabular}{l l}
\underline{\textbf{5. Step:}} & Show that the modified functions $\widetilde{g}^+_{k}$ and  $\widetilde{g}^-_{k}$ are smooth and psh on $V$.
\end{tabular}  
\vspace{0.2cm}

Note that both properties are local.
We already know that $g^\pm_{k}\in C^\infty(V_Y)$, so it remains to find for every point  $x$ in  $\tau^{-1}_{\XS'}(\Omega_{y,e})\cup V_e$ an open neighborhood $V_x$ in $V$ such that $\widetilde{g}^\pm_{k}$ is smooth and psh on $V_x$.

We start with a point $x\in V_e$.
Choose an open neighborhood $V_x$ of $x$ in the open subset $V_e\subset V$, then $\widetilde{g}^\pm_{k}$ is given by $\psi_{y,e}\circ\tau_{\XS'}$ on $V_x$. 
For every $\zeta \in V_x\subset A$, we have  
\begin{align*}
	\widetilde{g}^\pm_{k}(\zeta)=(\psi_{y,e}\circ\tau_{\XS'})(\zeta)=(\psi_{y,e}\circ\tau_{A})(\zeta)
	= \Phi_{y,e}^*(-m\cdot\log|T|+c)(\zeta)
\end{align*}
on $V_x$. 
We first show
$$-m\cdot\log|T|+c\in \ker(d'd''\colon C^\infty(\G_m^{1,\an})\to \AS^{1,1}(\G_m^{1,\an})).$$
Consider the tropical chart $(V,\varphi_U)=(\G_m^{1,\an},\id)$ of $\G_m^{1,\an}$. Then $\trop_U=\log|T|$, and so  $-m\cdot\log|T|$ can be written as the triple $(\G_m^{1,\an},\id,\lambda)$, where $\lambda\colon \R\to \R$ is the affine function $t\mapsto mt+c$.
Thus, $-m\cdot\log|T|+c$ is a smooth function on $\G_m^{1,\an}$ (cf.~Definition \ref{D forms}). The $(1,1)$ form $d'd''(-m\cdot\log|T|+c)$ is given by the triple $(\G_m^{1,\an},\id,d'd''\lambda)$. Since $\lambda$ is affine, the form $d'd''\lambda$ is zero, and so is $ d'd''(-m\cdot\log|T|+c)$.
Consequently,  $-m\cdot\log|T|+c\in \ker(d'd''\colon C^\infty(\G_m^{1,\an})\to \AS^{1,1}(\G_m^{1,\an}))$.
This implies that $\widetilde{g}^\pm_{k}|_{V_x}=\Phi_{y,e}^*(-m\cdot\log|T|+c)$ is a smooth psh function on $V_x$ (see Remark \ref{Pullback Bem}).
 
 Now, consider $x\in \tau_{\XS'}^{-1}(\Omega_{y,e})$.
 We have just seen that $\psi_{y,e}\circ \tau_{\XS'}$ is a smooth psh function on $\tau_{\XS'}^{-1}((y,x_e))$. 
 Using Proposition \ref{Prop smooth psh}, there is a chart $(V_x,\varphi_{U_x})$ with $x\in V_x\subset \tau_{\XS'}^{-1}((y,x_e))$ such that $\psi_{y,e}\circ \tau_{\XS'}=\phi\circ \trop_{U_x}$ on $V_x$ for a smooth function $\phi$ on $\R^{r}$ that is convex restricted to every  polyhedron  contained in $\trop_{U_x}(V_x)$. 
 On the other hand, we know that the function $g^\pm_{k}$ is smooth and psh on $\tau_{\XS'}^{-1}((y,x_e))$ as well. 
 Hence, there is also a chart $(V'_x,\varphi_{U'_x})$ with $x\in V'_x\subset \tau_{\XS'}^{-1}((y,x_e))$ such that $g^\pm_{k}=\phi'\circ \trop_{U'_x}$ on $V'_x$ for a smooth function $\phi'$ on $\R^{r'}$ that is convex restricted to every  polyhedron contained in $\trop_{U'_x}(V'_x)$.
 Working on the intersection $(V_x\cap V'_x,\varphi_{U_x}\times\varphi_{U'_x})$, which is a subchart of both \cite[Proposition 4.16]{Gubler}, we get 
 \begin{align*}
 \psi_{y,e}\circ \tau_{\XS'}&=(\phi\circ \Trop(\pi)) \circ \trop_{U_x\cap U'_x}\\
 g^\pm_{k}&=(\phi'\circ \Trop(\pi'))\circ \trop_{U_x\cap U'_x}	
 \end{align*} 
for the corresponding transition functions $\pi,\pi'$ satisfying $\varphi_{U_x}=\pi\circ (\varphi_{U_x}\times\varphi_{U'_x})$ and $\varphi_{U'_x}=\pi'\circ (\varphi_{U_x}\times\varphi_{U'_x})$.
 Since $\Trop(\pi)$ and $\Trop(\pi')$ are integral affine functions on $\R^{r+r'}$, the composition $\phi\circ \Trop(\pi)$ (resp.~$\phi'\circ \Trop(\pi')$) is still a smooth function on $\R^{r+r'}$ with a convex restriction to every polyhedron.
 Thus, $m_{\varepsilon}(\phi\circ \Trop(\pi),\phi'\circ \Trop(\pi'))$ is a smooth function on $\R^{r+r'}$, and the properties i) and iv) of $m_{\varepsilon}$ imply that the restriction to every polyhedron is convex since the restriction of $\phi\circ \Trop(\pi)$ and $\phi'\circ \Trop(\pi')$ are.
We have 
 $$\widetilde{g}^\pm_{k}=m_{\varepsilon}(g^\pm_{k},\psi_{y,e}\circ \tau_{\XS'})=m_{\varepsilon}(\phi\circ \Trop(\pi),\phi'\circ \Trop(\pi'))\circ \trop_{U_x\cap U'_x}$$ on $V_x\cap V'_x$ for every $x\in \tau_{\XS'}^{-1}(\Omega_{y,e})$, so $\widetilde{g}^\pm_{k}$ is a smooth psh function on $\tau_{\XS'}^{-1}(\Omega_{y,e})$ by Proposition \ref{Prop smooth psh}, which proves Step 5.
 
Altogether, $\widetilde{g}^\pm_{k}$  are smooth psh function on $V$.
We set 
$$\widetilde{g}_{k}:=\widetilde{g}^+_{k}-\widetilde{g}^-_{k},$$ 
then the functions $\widetilde{g}_{k}$ satisfy by construction all the required properties in Step 0.

\vspace{0.5cm}

\begin{tabular}{l l}
	\underline{\textbf{6. Step:}} & Construction of $V'$ and $\eta$.
\end{tabular}  

\vspace{0.3cm} 
 By the construction of $V$ and $\widetilde{g}_{k}$, it is no problem to construct the required set $V'$, i.e.~an open subset $V'$ of $V$ containing $Y$ such that $\overline{V'}\subset V$ and $\widetilde{g}_k\in C^\infty_c(V')$. 
 Furthermore, let $V''$ be an open neighborhood of $\overline{V'}$ in $V$ with $\overline{V''}\subset V$. 
 The topological space $\Xan$ is a compact Hausdorff space.
 Urysohn's Lemma states the existence of a continuous function $\eta\colon \Xan \to [0,1]$ with $\eta \equiv1$ on $\overline{V'}$ and $\eta \equiv0$ on $\Xan\backslash V''$.
 Thus, $\eta$ has compact support in $V$, i.e.~it is the required function in Step 0.
\vspace{0.5cm}

Thus, we have constructed everything as it was described in Step 0 proving the theorem.
\end{proof}

\begin{kor}\label{Kor Equ}
	Let $f\colon W\to\R$ be a continuous function. Then $f$ is subharmonic if and only if $f$ is psh.
\end{kor}
\begin{proof}
	Follows directly from Theorem \ref{Korollar Th impl. CLD} and Theorem \ref{folgt subh}.
\end{proof}

\section{Stability under pullback and a regularization theorem}
Due to this equivalence in Corollary \ref{Kor Equ}, we know that a continuous psh function has all the nice properties that are shown for subharmonic functions in \cite{Th}. 
More precisely, we now know that the property psh for continuous functions is stable under pullback with respect to morphisms of curves. 
Furthermore, we show that there is a monotone regularization theorem in the setting of Chambert-Loir and Ducros under certain conditions, e.g.~if $X$ is $\mathbb{P}_K^1$ or a Mumford curve.

\begin{kor}\label{Kor Pullback}
Let $X,X'$ be smooth proper algebraic curves over $K$ and $\varphi\colon W'\to W$ be a morphism of $K$-analytic spaces for open subsets $W\subset \Xan$ and $W'\subset (X')^{\an}$. If a continuous function $f\colon W\to \R$ is psh on $W$, then  $\varphi ^*f$ is psh on $\varphi ^{-1}(W)$. 
\end{kor}
\begin{proof}
By Corollary \ref{Kor Equ}, the function $f$ is subharmonic on $W$, and so  $\varphi^* f$ is subharmonic on $\varphi ^{-1}(W)$ by Proposition \ref{Prop Pullback}. Using again Corollary \ref{Kor Equ},  $\varphi^*f$ is psh on $\varphi^{-1}(W)$.
\end{proof}

To obtain a  monotone regularization theorem in the setting of Chambert-Loir and Ducros for certain $X$, e.g.~for the projective line or a Mumford curve, we   use the monotone regularization theorem in Thuillier's setting. 
Hence, we first need to show that every point has an open neighborhood such that every lisse subharmonic, and so psh, function  is globally psh-approximable on it, i.e.~it is the uniform limit of smooth psh functions. 
Recall the definition of the sheaf $C^\infty$ of smooth functions on $\Xan$ from Definition \ref{Def glatt}.
The key tool of this step is to use that for certain $X$ every lisse function $f$ with $dd^c f=0$ is automatically smooth (cf.~\cite[Th\'eor\`eme 2.3.21]{Th}).
These functions are the harmonic functions and the corresponding sheaf is denoted by $\HS_X$ (cf.~Definition \ref{Def harmonisch Berk}). Note that in general, every smooth function $f$ with $d'd''f=0$ is harmonic.

\begin{lem}\label{globally psh-approx}
Let $X$ be a smooth proper algebraic curve such that the sheaf $\HS_X$ of harmonic functions on $\Xan$ is a subsheaf of the sheaf  $C^\infty$ of smooth functions on $\Xan$.
Then every lisse function $f\colon Y\to \R$ on a strictly affinoid domain $Y$ of $\Xan$ which is subharmonic on the interior $W:=Y\backslash\partial Y$  of $Y$ is globally psh-approximable on $W$.
More precisely, there is a monotone decreasing sequence of continuous functions $f_k$ on Y that are smooth and psh on $W$, and converge uniformly to $f$ on $Y$.
\end{lem}

\begin{proof} 
Since $f$ is lisse, we can find a strictly semistable model $\YS$ such that $f=F\circ \tau_{\YS}$ on $Y$ for a piecewise affine function $F$ on $S(\YS)$.
We construct continuous functions on $Y$ converging uniformly to $f$ that are smooth and psh on $W$ using techniques as in the proof of Theorem \ref{folgt subh}.

Let $S$ be the set of points in $S(\YS)\backslash \partial Y$ that are contained in the support of the discrete measure $dd^c F$. 
Consider in the following a point $x$ in $S$.
Then $dd^c F> 0$  in an open neighborhood of $x$ because $f$ is subharmonic.
The considered point $x$ is either of type II or III. 
If $x$ is of type II, we may assume $x$ to be a vertex of $S(\YS)$ and we denote by $e_{x,1},\ldots,e_{x,n}$ the adjacent edges in $S(\YS)$ and by $x_i$ the second endpoint of $e_{x,i}$.
If $x$ is of type III, $x$ is contained in the interior of an edge $e_x$ with endpoints $x_1$ and $x_2$ and we denote by $e_{x,1}$ and $e_{x,2}$ the segments $[x_1,x]$ and $[x,x_2]$ of $e_x$. 
By blowing up $\YS$, we may assume that no $x_i$ belongs to $S$ and that $F$ restricted to every  $e_{x,i}$ is affine. 

In both situations, type II or III, we can find a piecewise affine function $G_x$ on $\Gamma_x:=\bigcup_{i=1,\ldots,n} e_{x,i}$ such that 
\begin{enumerate}
\item $G_x(x)=F(x)$,
\item $G_x< F$ on $\Gamma_x\backslash \{x\}$,
\item   $dd^c G_x=0$ in a neighborhood of $x$, and
\item $(G_x)|_{e_{x,i}}$ is affine for every $i=1,\ldots,n$.

\end{enumerate}    
Choose $\varepsilon_{x,i}>0$ with  $F(x_i)-G_x(x_i)> 2\varepsilon_{x,i}.$ 
Then there is a point $y_i\in (x,x_i)$ such that $F(y_i)= G_x(y_i)+\varepsilon_{x,i}$,  $F< G_x+\varepsilon_{x,i}$ on $[x,y_i)$ and $F> G_x+\varepsilon_{x,i}$ on $(y_i,x_i]$.
Since $S$ has only finitely many  points and corresponding adjacent edges, we can set $\varepsilon_0:=\min_{x\in S,i}\varepsilon_{x,i}$.
Then for every $x\in S$, every $e_{x,i}=[x,x_i]$, and every $0<\varepsilon\leq\varepsilon_0$ the following inequalities
 \begin{align}\label{U1}
G_x(x)+\varepsilon-F(x)>  \varepsilon/2, ~~~~
F(x_i)-(G_x(x_i)+\varepsilon)> \varepsilon/2 
\end{align}
hold.

For every $x\in S$, the set $V_x:=(\tau_{\YS}^{-1}(\Gamma_x^\circ))$  is an open neighborhood of $x$ in $W$ and
by construction these sets are pairwise disjoint. 
We define for every $0<\varepsilon\leq\varepsilon_0$ the following function on $Y$
\begin{align}\label{f_k}
f_{\varepsilon}:=
\begin{cases}
m_{\frac{\varepsilon}{2}}(G_x\circ \tau_{\YS}+\varepsilon,f) & \text{on } V_x \text{ for } x\in S,\\
f & \text{on } Y\backslash \bigcup_{x\in S} V_x,
\end{cases} 
\end{align}
where $m_{\varepsilon}$ is the smooth maximum defined in (\ref{m_eps}) (see proof of Theorem \ref{folgt subh}).
By $(\ref{U1})$ and  property iii)  of $m_{\varepsilon/2}$ in  (\ref{m_eps}), the defined function coincides with $G_x\circ \tau_{\YS}+\varepsilon$ in an open neighborhood of $x$ and with $f$ in an open neighborhood of $x_i$, where $x_i$ is the other vertex for an adjacent $e_{x,i}=[x,x_i]$.
Thus, $f_{\varepsilon}$ is continuous on $Y$.

We will later use functions of this form to construct our sequence, but first we  show that $f_{\varepsilon}$ is smooth and psh on $W$.
By construction, there is an open neighborhood $W'$ of $W\backslash \bigcup_{x\in S} V_x$ such that $f_{\varepsilon}$ coincides with $f$ and $f=F\circ \tau_{\YS}$ is harmonic on $W'$.
Since  we required that $\HS_X$ is a subsheaf of $C^\infty$ and every harmonic function is psh by Proposition \ref{Prop lisse subh} and Theorem \ref{folgt subh}, the function $f_{\varepsilon}$ is a smooth psh function on $W'$.

On the other hand, for every $x\in S$ the constructed function $f_{\varepsilon}$ coincides with the harmonic function $G_x\circ \tau_{\YS}+\varepsilon$ on an open neighborhood of $x$, and so it is locally smooth and psh  at $x$ as well.
It remains to consider $f_{\varepsilon}$ on $\tau^{-1}_{\YS}((x,x_i))$ for every $x\in S$ and every adjacent $e_{x,i}=[x,x_i]$.
Since $f$ and $G_x\circ \tau_{\YS}+\varepsilon$ are harmonic, and so smooth and psh on $\tau^{-1}_{\YS}((x,x_i))$, one can show as in Step 5 in  the proof of Theorem  \ref{folgt subh} that $f_{\varepsilon}=m_{\frac{\varepsilon}{2}}(G_x\circ \tau_{\YS}+\varepsilon,f)$ is still smooth and psh on $\tau^{-1}_{\YS}((x,x_i))$.
Altogether, $f_{\varepsilon}$ is a smooth psh function on $W$.

With the help of the function $f_{\varepsilon}$ defined in (\ref{f_k}), we construct now a monotonically decreasing sequence $(f_k)_{k\in \N}$ of smooth psh functions converging uniformly to $f$ on $Y$.
For every $k\in \N$, we define $\varepsilon_k >0$ recursively starting with $\varepsilon_0$ from above, and set $f_k:=f_{\varepsilon_k}$.
To do so, we need to consider the subsets
 $$\Omega_{k}:=\bigcup_{x\in S}\left\{y\in V_x\mid |G_x(\tau_{\YS}(y))+\varepsilon_k-f(y)|< \frac{\varepsilon_k}{2} \right\}$$ on which $f_k$ does not necessarily coincide with $\max(G_x\circ \tau_{\YS}+\varepsilon_k,f)$ for some $x\in S$.
 For a given $\varepsilon_k$, we choose $\varepsilon_{k+1}$ such that  $0< \varepsilon_{k+1}<\varepsilon_{k}/3$.  Then $\Omega_{k}\cap\Omega_{k+1}=\emptyset$ for every $k\in \N$ and $\varepsilon_k\to 0$ for $k\to \infty$.
 
We show that the sequence $f_k$ converges pointwise to $f$ and $f_{k+1}\leq f_k$ on $Y$.
If $y\in Y\backslash \bigcup_{x\in S}V_x$, then all $f_k$ coincide with $f$, and so both assertions are trivial. 
We therefore assume that $y\in V_x$ for some $x\in S$.
In the case of $y\in \tau^{-1}_{\YS}(x)$, we have $$f_k(y)=G_x(x)+\varepsilon_k=F(x)+\varepsilon_k=f(y)+\varepsilon_k,$$ and so 
$f_k(y)\geq f_{k+1}(y)$ and $f_k(y)$ converges to $f(y)$ for $k\to \infty$.
Now we consider the case $y\in V_x\backslash \{\tau^{-1}_{\YS}(x)\}$. 
Then we can find an $\varepsilon_N$ small enough such that $f(y)-(G_x(\tau_{\YS}(y))+\varepsilon_N)>\varepsilon_N/2$.
Hence, for every $k\geq N$ we have $f(y)-(G_x(\tau_{\YS}(y))+\varepsilon_k)>\varepsilon_k/2$, and so
\begin{align*}
	f_k(y)=\max(G_x\circ \tau_{\YS}+\varepsilon_k,f)=f(y).
\end{align*} 
Thus, $f_k(y)$ converges clearly to $f(y)$.
Next, consider an arbitrary $k\in \N$ and show $f_k(y)\geq f_{k+1}(y)$.
If $y\notin \Omega_k\cup \Omega_{k+1}$, then 
\begin{align*}
f_{k+1}(y)= 	\max(G_x(\tau_{\YS}(y))+\varepsilon_{k+1},f(y))\leq \max(G_x(\tau_{\YS}(y))+\varepsilon_k,f(y))=f_k(y)
\end{align*}since $\varepsilon_k> \varepsilon_{k+1}$.
If $y\in \Omega_k$, by the choice of $\varepsilon_{k+1}$, we have  $y\notin\Omega_{k+1}$.
Thus, 
\begin{align*}
	f_{k+1}(y)=\max(G_x(\tau_{\YS}(y))+\varepsilon_{k+1},f(y))\leq \max(G_x(\tau_{\YS}(y))+\varepsilon_k,f(y))\leq f_{k}(y),
\end{align*}
where the last inequality is true by property ii) following (\ref{m_eps}).
Finally, let $y\in \Omega_{k+1}$, and so $y\notin  \Omega_{k}$.
Then $G_x(\tau_{\YS}(y))+\varepsilon_{k}\geq f(y)$ as $\varepsilon_{k+1}< \varepsilon_k/3$, and so 
$$f_k(y)=\max(G_x(\tau_{\YS}(y))+\varepsilon_{k}, f(y)) =G_x(\tau_{\YS}(y))+\varepsilon_{k}\geq f(y)+\varepsilon_{k}/2$$
as  $y\notin  \Omega_{k}$.
By property ii) following (\ref{m_eps}), $\varepsilon_{k+1}< \varepsilon_{k}/3$ and the last inequality, we get
\begin{align*}
	f_{k+1}(y)&\leq \max(G_x(\tau_{\YS}(y))+\varepsilon_{k+1},f(y))+\varepsilon_{k+1}/4
	\\&\leq\max(G_x(\tau_{\YS}(y))+\varepsilon_{k},f(y)+\varepsilon_{k}/2) \\
	&\leq f_{k}(y).
\end{align*}

Thus, the sequence $f_k$ of continuous functions converges pointwise to the continuous function $f$ on $Y$ and $f_{k+1}\leq f_k$. 
Since $Y$ is compact, $f_k$ converge uniformly by Dini's theorem.
We have already seen above that every $f_k$ is smooth and psh on $W$. 
\end{proof}

Before we use this lemma to prove a monotone regularization theorem in the setting of Chambert-Loir and Ducros for Mumford curves, we recall the definition of these curves. 

\begin{defn}
	A smooth proper curve $X$ of genus $g\geq 1$ over $K$ is called a \emph{Mumford curve} if there is a semistable model $\XS$ such that all irreducible components of the special fibre are rational (cf.~\cite[Theorem 4.4.1]{BerkovichSpectral}).
\end{defn}

\begin{kor} Let $X$ be a smooth proper curve over $K$.
If $\widetilde{K}$ is algebraic over a finite field or $X$ is the projective line or a Mumford curve, then every continuous psh function $f\colon W\to \R$  on an open subset $W$ of $\Xan$ is locally psh-approximable.
More precisely, the sequence of smooth psh functions can be chosen monotonically decreasing.
\end{kor}
\begin{proof} 
	At first, note that in the given situation $\HS_X$ is a subsheaf of $C^\infty$ on $\Xan$, which  is a consequence of \cite[Th\'eor\`eme 2.3.21]{Th} (see \cite[Corollary 5.3.21]{Wa}).	
	We additionally use \cite[Theorem 4.4.1]{BerkovichSpectral} in the case of a Mumford curve.	
	Hence, we may apply Lemma \ref{globally psh-approx}.
	
To prove the corollary,	we have to show that every point $x$ in $W$ has an open neighborhood in $W$ such that $f$ is a uniform limit of smooth psh functions.
The continuous psh function $f$ is subharmonic by Theorem \ref{folgt subh}, and  we therefore can use Thuillier's monotone regularization theorem (see Proposition \ref{Prop Net}).  
We can  find for every  $x\in W $ a relatively compact neighborhood $W'$ of $x$ in $W$ and a decreasing net $\langle f_j\rangle$ of lisse subharmonic functions on $W'$ converging pointwise to $f$.
Let $Y$ be a strictly affinoid domain in $W'$ having $x$ in its interior $Y^\circ$. 
Then the decreasing net $\langle f_j\rangle$ converges uniformly to the continuous function $f$ on the compact set $Y$ by Dini's theorem. 
Thus, one can construct inductively a decreasing sequence of lisse subharmonic functions on $W'$ converging uniformly to $f$ on $Y$ and we write $(f_k)_{k\in \N}$ for this sequence.

We have seen in Lemma \ref{globally psh-approx} that each $f_k$ is the uniform limit of a decreasing sequence of smooth psh functions on $Y^\circ$.
Hence, we can choose a decreasing sequence of smooth psh functions on $Y^\circ$ converging uniformly to $f$.
\end{proof}

\begin{bem}
	Note that there are curves such that the sheaf $\HS_X$ of harmonic functions is not a subsheaf of the sheaf $C^\infty$ of smooth functions on $\Xan$. A counter example of such a curve can be constructed by the proof of \cite[Th\'eor\`eme 2.3.21]{Th} (see for example \cite[Corollary 5.3.23]{Wa}). For those curves we do not know whether every psh function is locally psh-approachable.
\end{bem}

\bibliographystyle{alpha}
\def\cprime{$'$}

\end{document}